\documentclass{article}

\usepackage{pb-xy,pb-diagram}
\usepackage{amssymb,amsmath,amsthm}
\usepackage[pdftex]{graphicx,color}  
\usepackage{enumerate}
 
 \theoremstyle{plain}
 \newtheorem{thm}{Theorem}[section]
 \newtheorem{prop}[thm]{Proposition}
 \newtheorem{lem}[thm]{Lemma}
 \newtheorem{cor}[thm]{Corollary}
 \newtheorem{conj}[thm]{Conjecture}
 \newtheorem*{thm*}{Theorem}
 \newtheorem*{prop*}{Proposition}
 \newtheorem*{lem*}{Lemma}
 \newtheorem*{cor*}{Corollary}
 \newtheorem*{conj*}{Conjecture}
 
 \theoremstyle{definition}
 \newtheorem{defn}{Definition}[section]
 
 \newtheorem{rem}{Remark}[section]

 \theoremstyle{plain}

\newcommand{\bbZ}{\ensuremath{\mathbb{Z}}}

\newcommand{\bbR}{\ensuremath{\mathbb{R}}}
\newcommand{\bbQ}{\ensuremath{\mathbb{Q}}}
\newcommand{\bbC}{\ensuremath{\mathbb{C}}}
\newcommand{\bbA}{\ensuremath{\mathbb{A}}}
\newcommand{\bbP}{\ensuremath{\mathbb{P}}}
\newcommand{\bbF}{\ensuremath{\mathbb{F}}}
\newcommand{\bfF}{\ensuremath{\mathbf{F}}}

\newcommand{\GL}{\ensuremath{\operatorname{GL}}}
\newcommand{\SL}{\ensuremath{\operatorname{SL}}}
\newcommand{\ed}{\ensuremath{\operatorname{ed}}}
\newcommand{\covdim}{\ensuremath{\operatorname{covdim}}}

\newcommand{\rank}{\ensuremath{\operatorname{rank}}}

\newcommand{\Stab}{\ensuremath{\operatorname{Stab}}}
\newcommand{\Tr}{\ensuremath{\operatorname{Tr}}}
\newcommand{\PGL}{\ensuremath{\operatorname{PGL}}}
\newcommand{\Pic}{\ensuremath{\operatorname{Pic}}}
\newcommand{\Div}{\ensuremath{\operatorname{Div}}}
\newcommand{\PSL}{\ensuremath{\operatorname{PSL}}}
\newcommand{\Aut}{\ensuremath{\operatorname{Aut}}}
\newcommand{\Out}{\ensuremath{\operatorname{Out}}}
\newcommand{\tildeAut}{\ensuremath{\widetilde{\operatorname{Aut}}}}
\newcommand{\Cox}{\ensuremath{\operatorname{Cox}}}
\newcommand{\Hom}{\ensuremath{\operatorname{Hom}}}

\newcommand{\Res}{\ensuremath{\operatorname{Res}}}
\newcommand{\Spec}{\ensuremath{\operatorname{Spec}}}

\newcommand{\im}{\ensuremath{\operatorname{im}}}

\newcommand{\Ind}{\ensuremath{\operatorname{Ind}}}
\newcommand{\calG}{\ensuremath{\mathcal{G}}}
\newcommand{\calO}{\ensuremath{\mathcal{O}}}
\newcommand{\calR}{\ensuremath{\mathcal{R}}}

\newcommand{\spaceHack}{\rule{0pt}{2.6ex} \rule[-1.2ex]{0pt}{0pt}}

\providecommand{\mat}[1]{\ensuremath{\bigl( \begin{smallmatrix} #1 \end{smallmatrix} \bigr)}}

\begin{document}
 
\title{Finite Groups of Essential Dimension $2$}
\author{Alexander Duncan \thanks{Department of Mathematics, University of British Columbia, Vancouver, BC, V6T 1Z2, Canada; duncan@math.ubc.ca}}

\maketitle

\begin{abstract}
We classify all finite groups of essential dimension 2 over an algebraically closed field of characteristic 0.
\end{abstract}

\smallskip


\section{Introduction}\label{sec:intro}

Let $k$ be an algebraically closed field of characteristic $0$ and let $G$ be a finite group.  A faithful $G$-variety is a variety with a faithful $G$-action.  A \emph{compression} is a $G$-equivariant dominant rational map of faithful $G$-varieties.  Given a faithful $G$-variety $X$ over $k$, the \emph{essential dimension of $X$}, denoted $\ed_k(X)$, is the minimum dimension of $Y$ over all compressions $X \dasharrow Y$ where $Y$ is a faithful $G$-variety over $k$.  The \emph{essential dimension of $G$}, denoted $\ed_k(G)$, is the maximum of $\ed_k(X)$ over all faithful $G$-varieties $X$ over $k$.

The major result of this paper is a classification of all finite groups of essential dimension $2$ over an algebraically closed field of characteristic $0$.

Loosely speaking, the essential dimension of a group is the minimal number of parameters required to describe its faithful actions.  Essential dimension was introduced by Buhler and Reichstein in \cite{BuhlerReichstein1997On-the-essentia}.  Their main interest was to determine how much a ``general polynomial of degree $n$'' can be simplified via non-degenerate Tschirnhaus transformations.  They showed that essential dimension of the symmetric group on $n$ letters, $\ed_k(S_n)$, is the minimal number of algebraically independent coefficients possible for a polynomial simplified in this manner.

The essential dimension of finite groups in general is of interest in Inverse Galois Theory.  Here one wants to construct polynomials over a field $k$ with a given Galois group $G$.  Ideally, one wants polynomials that parametrize all fields extensions with that group: the so-called \emph{generic polynomials} (see \cite{KemperMattig2000Generic-polynom} and \cite{JensenLedetYui2002Generic-polynom}).  The essential dimension of $G$ is a lower bound for the \emph{generic dimension} of $G$: the minimal number of parameters possible for a generic polynomial.

Essential dimension has been studied in the more general contexts of algebraic groups \cite{Reichstein2000On-the-notion-o}, algebraic stacks \cite{BrosnanReichsteinVistoli2007Essential-dimen}, and functors \cite{BerhuyFavi2003Essential-dimen}.  We restrict our attention to the case of finite groups in this paper.

If $H$ is a subgroup of $G$ then $\ed(H) \le \ed(G)$; a similar inequality fails for quotient groups \cite[Theorem 1.5]{MeyerReichstein2008Some_consequenc}.  The essential dimension of an abelian group is equal to its rank \cite[Theorem 6.1]{BrosnanReichsteinVistoli2007Essential-dimen}.  The essential dimensions of the symmetric groups, $S_n$, and alternating groups, $A_n$, are known for $n \le 7$ and bounds exist for higher $n$ (see \cite[Theorems 6.5 and 6.7]{BrosnanReichsteinVistoli2007Essential-dimen}, \cite[Proposition 3.6]{Serre2008Le-group-de-cre} and \cite{Duncan2010Essential-dimen}).  It is a deep result of Karpenko and Merkurjev \cite{KarpenkoMerkurjev2008Essential-dimen} that the essential dimension of a $p$-group is the minimal dimension of a faithful linear representation.

We use the notation $D_{2n}$ to denote the dihedral group of order $2n$.  Finite groups of essential dimension $1$ were classified by Buhler and Reichstein in their original paper; they are either cyclic or isomorphic to $D_{2n}$ where $n$ is odd.  There is a classification for infinite base fields by Ledet \cite{Ledet2007Finite-groups-o} (see also Remark \ref{rem:Ledet}), and for arbitrary base fields by Chu, Hu, Kang and Zhang \cite{ChuHuKangZhang2008Groups_with_ess}.

We review what is known about groups $G$ of essential dimension $2$.  If $G$ contains an abelian subgroup $A$ then $\rank(A) \le 2$.  The Sylow $p$-subgroups $G_p$ of $G$ can be described using the Karpenko-Merkurjev theorem: $G_p$ must be abelian for all $p$ odd, and groups $G_2$ must be of a very special form (see \cite[Theorems 1.2 and 1.3]{MeyerReichstein2008Some_consequenc}).  Any subgroup of $\GL_2(\bbC)$ or $S_5$ has essential dimension $\le 2$.

Finite groups of essential dimension $2$ with non-trivial centres were classified (implicitly) by Kraft, L\"otscher and Schwarz (see \cite{KraftSchwarz2007Compression-of-} and \cite{KraftLotscherSchwarz2009Compression-of-}).  They show that a finite group with a non-trivial centre has essential dimension $\le 2$ if and only if it can be embedded in $\GL_2(\bbC)$.  Their main interest was in \emph{covariant dimension}, a ``regular'' analog of essential dimension.  See also \cite{Reichstein2004Compressions-of} and \cite{Lotscher2008Application-of-}.

Our study of essential dimension uses the concept of a \emph{versal $G$-variety} (defined in section \ref{sec:prelim}).  These are simply models of the versal torsors seen in Galois cohomology.  We will often say a $G$-action is versal if it gives rise to a versal $G$-variety.  The key fact is that if $G$ is a finite group of essential dimension $n$ then there exists a versal unirational $G$-variety of dimension $n$.

To study essential dimension $2$, we only need to consider versal rational $G$-surfaces since unirational surfaces are always rational over an algebraically closed field of characteristic $0$.  Furthermore, one can $G$-equivariantly blow-down sets of exceptional curves to obtain a \emph{minimal model} of a smooth $G$-surface.  The minimal rational $G$-surfaces were classified by Manin \cite{Manin1967Rational-surfac} and Iskovskikh \cite{Iskovskih1979Minimal-models-} building on work by Enriques: they either possess conic bundle structures or they are del Pezzo surfaces.  The use of the Enriques-Manin-Iskovskikh classification for computing essential dimension was pioneered by Serre in his proof that $\ed_k(A_6)=3$ \cite[Proposition 3.6]{Serre2008Le-group-de-cre}.  Independently, Tokunaga \cite{Tokunaga2006Two-dimensional} has also investigated versal rational surfaces.

The dichotomy into conic bundle structures and del Pezzo surfaces is too coarse to easily identify exactly which groups occur.  Our current work was inspired by Dolgachev and Iskovskikh's \cite{DolgachevIskovskikh2006Finite-subgroup} finer classification of such groups.  Their goal was to classify conjugacy classes of finite subgroups of the Cremona group of rank $2$ (the group of birational automorphisms of a rational surface).  This problem has a long history.  The first classification was due to Kantor; an exposition of his results (with some corrections) can be found in Wiman \cite{Wiman1896Zur_Thorie_der_}.  Unfortunately, this early classification had several errors, and the conjugacy issue was not addressed.  More recent work on this problem include \cite{BayleBeauville2000Birational_invo}, \cite{deFernex2004On_planar_Cremo}, \cite{Zhang2001Automorphisms-o}, \cite{Beauville2007p-elementary-su} and  \cite{Blanc2007Finite-abelian-}. 

Recall that the automorphism group of the algebraic group $(\bbC^\times)^n$ is isomorphic to $\GL_n(\bbZ)$.  Our main theorem is as follows:

\begin{thm}\label{thm:ed2classification}
Let $T = (\bbC^\times)^2$ be a $2$-dimensional torus.  If $G$ is a finite group of essential dimension $2$ then $G$ is isomorphic to a subgroup of one of the following groups:
\begin{enumerate}
\item $\GL_2(\bbC)$, the general linear group of degree $2$,
\label{thm:ed2:GL2}
\item $T \rtimes \calG_1$ with $|G \cap T|$ coprime to $2$ and $3$\\
$\calG_1 = \left\langle \mat{1&-1\\1&0}, \mat{0&1\\1&0} \right\rangle \simeq D_{12}$,
\label{thm:ed2:G1}
\item $T \rtimes \calG_2$ with $|G \cap T|$ coprime to $2$\\
$\calG_2 = \left\langle \mat{-1&0\\0&1}, \mat{0&1\\1&0} \right\rangle \simeq D_{8}$,
\label{thm:ed2:G2}
\item $T \rtimes \calG_3$ with $|G \cap T|$ coprime to $3$\\
$\calG_3 = \left\langle \mat{0&-1\\1&-1}, \mat{0&-1\\-1&0} \right\rangle \simeq S_3$,
\label{thm:ed2:G3}
\item $T \rtimes \calG_4$ with $|G \cap T|$ coprime to $3$\\
$\calG_4 = \left\langle \mat{0&-1\\1&-1}, \mat{0&1\\1&0} \right\rangle \simeq S_3$,
\label{thm:ed2:G4}
\item $\PSL_2(\bbF_7)$, the simple group of order $168$,
\label{thm:ed2:PSL27}
\item $S_5$, the symmetric group on $5$ letters.
\label{thm:ed2:S5}
\end{enumerate}
Furthermore, any finite subgroup of these groups has essential dimension $\le 2$.
\end{thm}

A few remarks are in order.
\begin{rem}\label{rem:versalGroupsNotSurfaces}
We do not classify all versal minimal rational $G$-surfaces; we only determine which groups appear.   Different $G$-surfaces with the same group $G$ may not be equivariantly birationally equivalent.  There exist two versal $S_5$-surfaces that are not equivariantly birationally equivalent: the Clebsch diagonal cubic (by a result of Hermite, see \cite{Coray1987Cubic_hypersurf}, \cite{ReichsteinYoussin2002Conditions_sati} and \cite{Kraft2006A_result_of_Her}) and the del Pezzo surface of degree $5$ (see the proof of Theorem \ref{thm:4surfacesClassification}).  Other examples of this phenomenon can be found for abelian groups \cite{ReichsteinYoussin2002A-birational-in}, and for versal actions of $S_4$ and $A_5$ \cite{BannaiTokunaga2007A-note-on-embed}.
\end{rem}

\begin{rem}
Essential dimension can be defined over any field.  Dolgachev and Iskovskikh's classification, and many of our other references, take the base field to be $\bbC$.  We shall see in Lemma \ref{lem:Csuffices} below that this is sufficient to handle any algebraically closed field of characteristic $0$.
\end{rem}

\begin{rem}
For algebraically closed fields of non-zero characteristic, the Enriques-Manin-Iskovskikh classification still holds.  However, the Dolgachev-Iskovskikh classification no longer applies.  Furthermore, unirational surfaces are not necessarily rational in this case, so the classification may be inadequate.  See \cite{Serre2008Le-group-de-cre} for related discussion.
\end{rem}

\begin{rem}\label{rem:generalChar0Fields}
For non-algebraically closed fields of characteristic $0$, we know that any group of essential dimension $2$ must appear in the list above (this is immediate from \cite[Proposition 1.5]{BerhuyFavi2003Essential-dimen}).  However, the problem of determining which groups appear is more complicated.  It is possible that a versal $G$-surface over a field $k$ may not be defined over a subfield $k'$ while there may be another versal $G$-surface that \emph{is} defined over $k'$.  A full classification of versal minimal rational $G$-surfaces would remedy this situation.
\end{rem}

\begin{rem}
For essential dimension $3$, one might try to do something similar with threefolds.  The problem is significantly more difficult.  First, even over $\bbC$ there exist unirational threefolds that are not rational.  Second, there is no analog of the Enriques-Manin-Iskovskikh classification here, nor the Dolgachev-Iskovskikh classification.  In fact, until recently it was an open question as to whether \emph{all} finite groups could be embedded into the Cremona group of rank $3$ \cite[6.0]{Serre2009A-Minkowski-sty}.

However, Prokhorov \cite{Prokhorov2009Simple-finite-s} shows that very few simple non-abelian groups can act faithfully on unirational threefolds.  The author \cite{Duncan2010Essential-dimen} has applied Prokhorov's work to show that the essential dimensions of $A_7$ and $S_7$ are $4$.
\end{rem}

\begin{rem}
Note that, since unirational and rational coincide in dimension $2$, for every group $G$ appearing above, there exists a versal $G$-variety $X$ whose rational quotient $X/G$ is rational.  This has consequences related to Noether's problem.  As suggested by the referee, for any faithful linear representation $V$ of a group $G$ in this list, the invariant field $\bbC(V)^G$ is retract rational (see \cite{Saltman1982Generic_Galois_}, \cite{DeMeyer1983Generic_polynom} or \cite[Remark 5(a)]{KemperMattig2000Generic-polynom}).  In addition, any such $G$ possesses a generic polynomial with only two parameters.  Thus, the list above is also a complete classification of groups of generic dimension $2$.
\end{rem}

The proof of Theorem \ref{thm:ed2classification} breaks into two mostly independent pieces.  We show that it suffices to consider only four surfaces:

\begin{thm}\label{thm:ed2to4surfaces}
If $G$ is a finite group of essential dimension $2$ then $G$ has a versal action on the projective plane $\bbP^2$, the product of projective lines $\bbP^1 \times \bbP^1$, or the del Pezzo surfaces of degree $5$ and $6$.
\end{thm}

And we show that the groups with versal actions on these four surfaces are those listed above (Theorem \ref{thm:4surfacesClassification}).  As in remark \ref{rem:versalGroupsNotSurfaces}, we point out that Theorem \ref{thm:ed2to4surfaces} does \emph{not} classify minimal versal $G$-surfaces.

We also mention some intermediary results that we feel are of independent interest.  Three of the four surfaces are toric varieties.  In order to classify their versal actions, we develop techniques that apply to smooth complete toric varieties in general.  We leverage the theory of \emph{Cox rings} \cite{Cox1995The-homogeneous} and \emph{universal torsors} \cite{ColliotTheleneSansuc1987La_descente_sur}): a faithful $G$-action on a complete non-singular toric variety is versal if and only if it lifts to an action on the variety of the associated Cox ring (Theorem \ref{thm:VersalCoxSplit}).

This result has some important corollaries.  First, if a complete non-singular toric variety has a $G$-fixed point then it is versal (Corollary \ref{cor:fixedPointImpliesVersal}).  Second, a complete non-singular toric variety is $G$-versal if and only if it is $G_p$-versal for all of its $p$-subgroups (Corollary \ref{cor:versalByPGroups}).  The assumption that the variety is toric may be gratuitous (see Conjecture \ref{conj:versalByPGroups}).  This second corollary is instrumental in our proof of Theorem \ref{thm:4surfacesClassification}; it reduces the study of versal toric surfaces to actions of $3$-groups on $\bbP^2$ and actions of $2$-groups on $\bbP^1 \times \bbP^1$.

The rest of this paper is structured as follows.  In section \ref{sec:prelim}, we recall basic facts about versal varieties, essential dimension and the Enriques-Manin-Iskovskikh classification.  In section \ref{sec:toricVersal}, we develop tools for determining when a toric $G$-variety is versal.  In section \ref{sec:4surfaces}, we determine precisely which groups act versally on the four surfaces of Theorem \ref{thm:ed2to4surfaces}.  In section \ref{sec:conics}, we show that all groups acting versally on conic bundle structures already act versally on the four surfaces.  In section \ref{sec:DelPezzo}, we show the same for the del Pezzo surfaces. This proves Theorem \ref{thm:ed2to4surfaces} and, thus, Theorem \ref{thm:ed2classification}.

\section{Preliminaries}\label{sec:prelim}

Recall that the main theorem applies for any algebraically closed field of characteristic $0$.  Nevertheless, for the rest of the paper, we will restrict our attention to $\bbC$.  This is possible in view of the following lemma:

\begin{lem}\label{lem:Csuffices}
Suppose $G$ is a finite group and $k$ is an algebraically closed field of characteristic $0$.  Then $\ed_k(G) = \ed_\bbC(G)$.
\end{lem}

\begin{proof}
This is just \cite[Proposition 2.14(1)]{BrosnanReichsteinVistoli2007Essential-dimen} since $k$ and $\bbC$ both contain an algebraic closure of $\bbQ$.
\end{proof}

We will make no more reference to a general field $k$.  All varieties, group actions and maps will be defined over $\bbC$ unless it is explicitly stated otherwise.  We write $\ed(-)$ instead of $\ed_\bbC(-)$ for the rest of the paper without risk of ambiguity.

\subsection{Versal Varieties}

\begin{defn}
An irreducible $G$-variety $X$ is \emph{$G$-versal} (or just \emph{versal}) if it is faithful and, for any faithful $G$-variety $Y$ and any non-empty $G$-invariant open subset $U$ of $X$, there exists a $G$-equivariant rational map $f : Y \dasharrow U$.  We say an action of $G$ is \emph{versal}, or that $G$ \emph{acts versally}, if the corresponding $G$-variety is versal.
\end{defn}

Note that the versal property is a birational invariant: it is preserved by equivariant birational equivalence.  In fact, our definition of versal variety is equivalent to saying that its generic point is a ``versal torsor'' as in \cite[Definition 6.3]{BerhuyFavi2003Essential-dimen} or \cite[Definition 5.1]{GaribaldiMerkurjevSerre2003Cohomological-i}.

Versal varieties are useful for studying essential dimension.  If $X$ is a versal $G$-variety then $\ed(X)=\ed(G)$ \cite[Corollary 6.16]{BerhuyFavi2003Essential-dimen}.  If $X \dasharrow Y$ is a compression of faithful $G$-varieties and $X$ is versal then so is $Y$ \cite[Corollary 6.14]{BerhuyFavi2003Essential-dimen}.  Thus, if a versal variety exists, there exists a versal variety $X$ such that $\dim(X)=\ed(X)=\ed(G)$.

Recall that a \emph{linear $G$-variety} is a linear representation of $G$ regarded as a $G$-variety.  Any faithful linear $G$-variety is versal \cite[Example 5.4]{GaribaldiMerkurjevSerre2003Cohomological-i}.  Thus versal varieties exist.  In particular, the essential dimension of any finite group is bounded above by the dimension of a faithful linear representation.

The versal property descends to subgroups:

\begin{prop}\label{prop:versalSubgroups}
Suppose $H$ is a subgroup of a finite group $G$.  If $X$ is a $G$-versal variety then $X$ is $H$-versal.
\end{prop}

\begin{proof}
Clearly, a faithful $G$-action restricts to a faithful $H$-action.  Consider any faithful $H$-variety $Y$ and any non-empty $H$-invariant open subset $U$ of $X$.  We need to show the existence of an $H$-equivariant rational map $f : Y \dasharrow U$.  The set $U'=\cap_{g \in G} g(U)$ is a $G$-invariant dense open subset of $U$.  Since $X$ is $G$-versal, there exists a $G$-equivariant rational map $\psi : V \dasharrow U'$ from a faithful linear $G$-variety $V$.  Let $W$ be a non-empty $H$-invariant open subset on which $\psi$ is defined.  Note that $V$ is $H$-versal since the restricted action still acts linearly.  Thus there exists an $H$-equivariant rational map $\phi : Y \dasharrow W$.  By composition, we obtain an $H$-equivariant map $f : Y \dasharrow U$ as desired.
\end{proof}

The following result is one of our major tools:

\begin{prop}\label{prop:abelSubgroups}
Let $G$ be a finite group.  If $X$ is a proper versal $G$-variety then all abelian subgroups of $G$ have fixed points on $X$.
\end{prop}

\begin{proof}
Note that the origin is a smooth fixed point of any linear $G$-variety $V$.  Thus, the result follows immediately by ``going down'' (\cite[Proposition A.2]{ReichsteinYoussin2000Essential-dimen}).
\end{proof}

We recall various standard results on essential dimension which can be found in \cite{BuhlerReichstein1997On-the-essentia}.  We say that a dihedral group, $D_{2n}$, of order $2n$ is an \emph{odd dihedral group} if $n$ is odd, and an \emph{even dihedral group} otherwise.

\begin{prop}\label{prop:edResults}
Let $G$ be a finite group.
\begin{enumerate}[(a)]
\item \label{prop:edResults:subgroups}
If $H$ is a subgroup of $G$ then $\ed(H) \le \ed(G)$.
\item \label{prop:edResults:abelian}
If $G$ is abelian then $\ed(G) = \rank(G)$.
\item \label{prop:edResults:ed0}
$\ed(G)=0$ if and only if $G$ is trivial.
\item \label{prop:edResults:ed1}
$\ed(G)=1$ if and only if $G$ is cyclic or odd dihedral.
\end{enumerate}
\end{prop}

The \emph{covariant dimension} of a group $G$, denoted $\covdim(G)$, is the minimal dimension of a faithful $G$-variety $X$ such that there is a faithful linear $G$-variety $V$ and a dominant regular $G$-equivariant map $V \to X$.  One may consider covariant dimension as a regular analog of essential dimension.  The interested reader is directed to the work of Kraft, L\"otscher and Schwarz (\cite{KraftSchwarz2007Compression-of-}, \cite{KraftLotscherSchwarz2009Compression-of-}).  The following result follows from the classification of groups of covariant dimension $2$.  We do not use the concept of covariant dimension anywhere else in this paper.  

\begin{prop}\label{prop:centreMeansP2}
If $G$ is a finite group of essential dimension $2$ with a non-trivial centre then $G$ is isomorphic to a subgroup of $\GL_2(\bbC)$.  In particular, $G$ has a versal action on $\bbP^2$.
\end{prop}

\begin{proof}
By \cite[Proposition 3.6]{KraftLotscherSchwarz2009Compression-of-}, whenever $G$ has a non-trivial centre we have $\ed(G)=\covdim(G)$.  By \cite[Section 7]{KraftLotscherSchwarz2009Compression-of-}, all finite groups of covariant dimension $2$ are isomorphic to subgroups of $\GL_2(\bbC)$.  Thus we have a faithful linear $G$-variety of dimension $2$.  This is versal and equivariantly birational to $\bbP^2$.
\end{proof}

We remark that, since all non-trivial $p$-groups have non-trivial centres, this proposition suffices to prove the Karpenko-Merkurjev theorem for groups of essential dimension $2$.  Recalling that all irreducible representations of $p$-groups have degree a power of $p$, we have the following:

\begin{prop}\label{prop:3groupsAbelian}
If $p > 2$ is a prime, then all $p$-groups of essential dimension $2$ are abelian.
\end{prop}

\subsection{Minimal Rational Surfaces}\label{sec:prelim:minimalModels}

We recall some basic facts about minimal rational surfaces (see \cite{DolgachevIskovskikh2006Finite-subgroup} or \cite{Manin1986Cubic-forms}).  Throughout this paper, a \emph{surface} is an irreducible non-singular projective $2$-dimensional variety over $\bbC$.  A \emph{minimal $G$-surface} is a faithful $G$-surface $X$ such that any birational regular $G$-map $X \to Y$ to another faithful $G$-surface $Y$ is an isomorphism.  There is a (not necessarily unique) minimal $G$-surface in every equivariant birational equivalence class of $G$-surfaces.  The possible minimal $G$-surfaces are classified as follows:

\begin{thm}[Enriques, Manin, Iskovskikh]\label{thm:EMI}
If $X$ is a minimal rational $G$-surface then $X$ admits a conic bundle structure or $X$ is isomorphic to a del Pezzo surface.
\end{thm}

Our interest in minimal rational surfaces is justified by the following proposition (see \cite[\S 3.6]{Serre2008Le-group-de-cre}):

\begin{prop}\label{prop:ratMinModelStrategy}
Suppose $G$ is a finite group.  Then $\ed(G) \le 2$ if and only if there exists a minimal rational versal $G$-surface $X$.
\end{prop}

\subsection{Polyhedral Groups}\label{sec:prelim:polyhedral}

The following facts will be used extensively in the sections that follow.  Most of these results can be found in, for example, \cite{Coxeter1991Regular-complex}.  Recall that a \emph{polyhedral group} is a finite subgroup of $\PGL_2(\bbC)$.  Equivalently, the polyhedral groups are precisely the finite groups acting regularly on $\bbP^1$.  The polyhedral groups were classified by Klein as follows: $C_n$, the cyclic group of $n$ elements; $D_{2n}$, the dihedral group of order $2n$; $A_4$, the alternating group on $4$ letters; $S_4$, the symmetric group on $4$ letters; and $A_5$, the alternating group on $5$ letters.

These groups have normal structures as follows:

\begin{prop}\label{prop:polyNormalStructure}
Suppose $P$ is a polyhedral group and $N$ is a non-trivial proper normal subgroup of $P$.  We have the following possibilities:
\begin{enumerate}[(a)]
\item $P \simeq S_4$, $N \simeq A_4$, $P/N \simeq C_2$,
\item $P \simeq S_4$, $N \simeq C_2 \times C_2$, $P/N \simeq S_3$,
\item $P \simeq A_4$, $N \simeq C_2 \times C_2$, $P/N \simeq C_3$,
\item $P \simeq D_{2n}$, $N \simeq C_{m}$, $P/N \simeq D_{2n/m}$ where $m | n$,
\item $P \simeq D_{4n}$, $N \simeq D_{2n}$, $P/N \simeq C_2$,
\item $P \simeq C_{n}$, $N \simeq C_m$, $P/N \simeq C_{m/n}$ where $m | n$.
\end{enumerate}
Note that $D_2 \simeq C_2$ and $D_4 \simeq C_2 \times C_2$ are included above as degenerate cases.
\end{prop}

Finally, we will need the following fact about lifts of polyhedral groups to $2$-dimensional representations:

\begin{prop}\label{prop:PGL2toGL2}
A finite subgroup $G$ of $\PGL_2(\bbC)$ has an isomorphic lift in $\GL_2(\bbC)$ if and only if $G$ is cyclic or odd dihedral.
\end{prop}

\section{Versal Actions on Toric Varieties}\label{sec:toricVersal}

In section \ref{sec:4surfaces} we will use the theory of toric varieties extensively to prove Theorem \ref{thm:4surfacesClassification}.  Many of the results we use are applicable beyond the case of surfaces so we consider the case of versal actions on toric varieties in general.

\subsection{Cox Rings and Universal Torsors}

We recall the theory of toric varieties from \cite{Fulton1993Introduction-to}, and Cox rings from \cite{Cox1995The-homogeneous}.  We will also use the language of universal torsors from \cite{ColliotTheleneSansuc1987La_descente_sur}.  Note that the similarity of the terms ``universal torsor'' and ``versal torsor'' is merely an unfortunate coincidence.

Given a lattice $N \simeq \bbZ^n$, a \emph{fan} $\Delta$ in $N$ is a set of strongly convex rational polyhedral cones in $N\otimes \bbR$ such that every face of a cone in $\Delta$ is in $\Delta$ and the intersection of any two cones in $\Delta$ is a face of each.  Given a fan one may construct an associated toric variety.

The associated toric variety $X = X(\Delta)$ contains an $n$-dimensional torus $T = N \otimes \bbC^{\times}$.  The variety $X$ is non-singular if every cone in $\Delta$ is generated by a subset of a basis for $N$.  The variety $X$ is complete if the support of the fan is all of $N \otimes \bbR$.  In this paper, we will restrict our attention to complete non-singular toric varieties.

Let $M = \Hom(N,\bbZ)$ be the dual of the lattice $N$.  Let $\Div_T(X)$ be the group of $T$-invariant divisors of $X$.  Let $\Delta(1)$ be the set of rays in the fan $\Delta$.  To each ray $\rho \in \Delta(1)$ we may associate a unique prime $T$-invariant divisor $D_\rho$.  In fact, $\Delta(1)$ is a basis for $\Div_T(X)$.  We have the following exact sequence:
\[ 1 \to M \to \Div_T(X) \to \Pic(X) \to 1 \]
where $\Pic(X)$ is the Picard group of $X$.

Denote $K = \Hom(\Pic(X),\bbC^\times)$ and apply $\Hom(\cdot,\bbC^\times)$ to the above sequence to obtain another exact sequence:
\[ 1 \to K \to (\bbC^\times)^{\Delta(1)} \to T \to 1 \ . \]
From \cite{Cox1995The-homogeneous}, any toric variety $X$ has an associated \emph{total coordinate ring}, or \emph{Cox ring}, which we denote $\Cox(X)$.  The ring $\Cox(X)$ is a $\Pic(X)$-graded polynomial ring
\[ \Cox(X) = \bbC[ x_\rho : \rho \in \Delta(1)] \]
and has an induced $K$-action via the grading.  (Note that Cox uses $G$ where we write $K$).

The variety $V = \Spec(\Cox(X))$ is isomorphic to affine space $\bbC^{\Delta(1)}$ and there is a closed subset $Z \subset V$ obtained from an ``irrelevant ideal.''  The open subset $V-Z$ is invariant under the $K$-action and, since $X$ is non-singular, the map $(V-Z) \to X$ is a $K$-torsor.  Indeed, this torsor is a \emph{universal torsor} over $X$.

We define $\tildeAut(X)$ as the normaliser of $K$ in the automorphism group of $V - Z$.  From \cite[Theorem 4.2]{Cox1995The-homogeneous}, there is an exact sequence
\[ 1 \to K \to \tildeAut(X) \to \Aut(X) \to 1 \]
where we denote the last map $\pi : \tildeAut(X) \to \Aut(X)$.

Let $\Aut(N,\Delta)$ denote the subgroup of $\GL(N)$ which preserves the fan $\Delta$ (permutes the cones).  The group $\Aut(N,\Delta)$ has an isomorphic lift to $\tildeAut(X)$ via permutations of the basis elements $\{x_\rho\}$.  The group $(\bbC^\times)^{\Delta(1)}$ is a subgroup of $\tildeAut(X)$ which descends to $T \subset \Aut(X)$.

More generally, if $G$ is a group with a faithful action on $X$ then there is a group
\[ E = \pi^{-1}(G) \subset \tildeAut(X) \]
acting faithfully on $V - Z$.  We have an exact sequence of groups
\begin{equation}\label{eqn:KEGseq}
1 \to K \to E \to G \to 1.
\end{equation}

For finite groups $G$, the group $E$ acts as a subgroup of $\GL(V)$:

\begin{lem}
Let $G$ be a finite group acting faithfully on $X$.  Then $E$ acts linearly on $V$.
\end{lem}

\begin{proof}
From \cite[Section 4]{Cox1995The-homogeneous}, the linear algebraic group $\tildeAut(X)$ is of the form $(R_u \rtimes G_s) \cdot \Aut(N,\Delta)$ where $R_u$ is unipotent and $G_s$ is reductive.  (Note that Cox's notation $G_s$ has nothing to do with the group $G$ in our context).  Since $G$ is finite and $K$ consists of semisimple elements, all elements of $E$ are semisimple.  Thus $E \subset G_s \cdot \Aut(N,\Delta)$.  The group $G_s$ is of the form
\[ G_s = \prod \GL(S_{\alpha_i}')\]
where the $S_{\alpha_i}'$s are the weight-spaces of the action of $K$ on $V$ (as a vector space).  The group $\Aut(N,\Delta)$ permutes the basis vectors of $V$.  Thus $G_s \cdot \Aut(N,\Delta)$ acts linearly on $V$.  Thus the subgroup $E$ acts linearly.
\end{proof}

The versal property is related to Cox rings by the following result:

\begin{thm}\label{thm:VersalCoxSplit}
Suppose $G$ is a finite group and $X$ is a complete non-singular toric faithful $G$-variety.  Then $X$ is versal if and only if the exact sequence (\ref{eqn:KEGseq}) splits.
\end{thm}

\begin{proof}
Suppose the exact sequence splits.  The map from $V - Z$ to $X$ may be viewed as a dominant rational map $\psi : V \dasharrow X$.  Since $E$ is linear for any finite group $G$, we obtain an $E$-equivariant rational map from a linear $E$-variety to $X$.  Since there is a section $G \hookrightarrow E$ the map $\psi$ may be viewed as a $G$-equivariant dominant rational map from a linear $G$-variety.  Thus $X$ is versal.

For the other implication, we assume $X$ is versal and want to show (\ref{eqn:KEGseq}) splits.

Since $X$ is versal there exists a $G$-equivariant rational map $f:W \dasharrow X$ where $W$ is a faithful linear $G$-variety.  Let $P \to U$ be the $K$-torsor obtained by pulling back $\psi$ along the restriction of $f$ to its domain of definition $U$.  From the universal property of pullbacks we obtain an $E$-action on $P$ compatible with the $G$-action on $U$.

Note that $\Pic(U)=0$ since $U$ is open in the affine space $W$.  Thus, from the exact sequence \cite[(2.0.2)]{ColliotTheleneSansuc1987La_descente_sur}, we see that the \'etale cohomology group $H^1(U,K)$ is trivial.  In particular, the torsor $P \to U$ is trivial.

Since $X$ is proper and $U$ is normal, the indeterminacy locus of $f$ is of codimension $\ge 2$.  Thus, all invertible global functions on $U$ are constant and the space of sections of $P \to U$ is isomorphic to $K$.  Thus $E$ has an induced action on $K$ and the desired splitting follows from Lemma \ref{lem:actionOnTorsorsSplit} below.
\end{proof}

\begin{lem}\label{lem:actionOnTorsorsSplit}
Suppose $E$ is an algebraic group with closed normal subgroup $K$ and quotient $G = E/K$.  Suppose $E$ acts on $K$ such that the restricted action of $K$ on itself is translation.  Then $E$ splits as $K \rtimes G$.
\end{lem}

\begin{proof}
Take any point $p \in K$ and consider the stabiliser $S = \Stab_E(p)$.  For any $g \in G$ we have a lift $h \in E$.  There is an element $k \in K$ such that $kh(p)=p$.  Thus $kh \in S$ and it follows that $S/(S \cap K) = G$.  Since $K$ acts freely on itself, $S \cap K = 1$.  Thus $S \simeq G$ and we have a splitting of $E$.
\end{proof}

\begin{rem}\label{rem:generalCoxRings}
Cox rings and universal torsors can be defined in more generality than the context of toric varieties (see, for example, \cite{ColliotTheleneSansuc1987La_descente_sur}, \cite{HuKeel2000Mori_dream_spac} and \cite{LafaceVelasco2009A-survey-on-Cox}).  The Cox rings of minimal rational surfaces have been extensively studied.  For example, conic bundles are considered in \cite[\S 2.6]{ColliotTheleneSansuc1987La_descente_sur}; del Pezzo surfaces, in \cite{Derenthal2007Universal_torso} and \cite{SerganovaSkorobogatov2007Del_Pezzo_surfa}.  It would be interesting to investigate versality using these constructions.

In fact, the proof of Theorem \ref{thm:VersalCoxSplit} still applies in one direction: if a $G$-action on a complete non-singular variety is versal then an analogous exact sequence to (\ref{eqn:KEGseq}) would still split.  However, the analog of $V$ is linear if and only if $X$ is toric \cite[Corollary 2.10]{HuKeel2000Mori_dream_spac}.  Thus, in general, one does not have an obvious compression from a linear $G$-variety as above.
\end{rem}

Recall that the standard projection $\bbC^n \dasharrow \bbP^{n-1}$ is an example of $\psi
$ obtained from the Cox ring.  We point out a special case of the preceding proposition:

\begin{cor}\label{cor:versalPn}
Let $G$ be a finite group acting faithfully on $X=\bbP^{n-1}$.  Then $X$ is $G$-versal if and only if there exists an embedding $G \hookrightarrow \GL_n(\bbC)$ such that the composition with the canonical map $\GL_n(\bbC) \to \PGL_n(\bbC)$ gives the $G$-action on $\bbP^{n-1}$.
\end{cor}

\begin{rem}\label{rem:Ledet}
Theorem \ref{thm:VersalCoxSplit} and Corollary \ref{cor:versalPn} were inspired by Ledet's classification of finite groups of essential dimension $1$ over an infinite ground field $k$ \cite{Ledet2007Finite-groups-o}.  Indeed, \cite[Theorem 1]{Ledet2007Finite-groups-o} states that a finite group $G$ has essential dimension $1$ if and only if there is an embedding $G \hookrightarrow \GL_2(k)$ such that the image of $G$ contains no scalar matrices $\ne 1$.  Such a group descends isomorphically to a subgroup of $\PGL_2(k)$.  In other words, the action of $G$ on $\bbP^1_k$ lifts to $\bbA^2_k$.
\end{rem}

The following is a useful tool for showing that a variety is versal.

\begin{cor}\label{cor:fixedPointImpliesVersal}
Suppose $G$ is a finite group acting faithfully on a complete non-singular toric variety $X$.  If $G$ has a fixed point then $X$ is $G$-versal.
\end{cor}

\begin{proof}
We have an action of $E$ on the fibre of the fixed point, so the result follows from Lemma \ref{lem:actionOnTorsorsSplit}.
\end{proof}

\begin{rem}\label{rem:fixedPointNotAlwaysVersal}
We note that Corollary \ref{cor:fixedPointImpliesVersal} fails when $X$ is not toric.  Consider a hyperelliptic curve $C$ with its involution (generating a group $G \simeq C_2$).  This is a faithful $G$-variety with a fixed point.  However, $C$ cannot be versal since the image of a rational map from a linear variety to $C$ must be a point.
\end{rem}

The following corollary is also inspired by Ledet \cite{Ledet2007Finite-groups-o}:

\begin{cor}\label{cor:versalByPGroups}
Let $G$ be a finite group acting faithfully on a complete non-singular toric variety $X$.  Then $X$ is $G$-versal if and only if, for any prime $p$, $X$ is $G_p$-versal for a Sylow $p$-subgroup $G_p$ of $G$.
\end{cor}

\begin{proof}
Using Theorem \ref{thm:VersalCoxSplit} this follows from a well-known result in group cohomology.  Consider the product $\prod_p \Res^G_{G_p}$ of the restriction maps $\Res^G_{G_p} : H^2(G,K) \to H^2(G_p,K)$ over all primes $p$ and some choice of Sylow $p$-subgroups $G_p$ for each $p$.  From \cite[Section III.10]{Brown1982Cohomology-of-g}, this product is an injection.  Thus $E \to G$ has a section if and only if every $G_p$ has a section.
\end{proof}

We remark on one application of this corollary that is not immediately obvious, but extremely useful.  Suppose $X$ is a $G$-variety and we want to determine whether or not it is versal.  For each prime $p$, let $G_p$ be a $p$-Sylow subgroup of $G$.  Suppose $X$ is $G_p$-equivariantly birational to a $G_p$-variety $X_p$ for each prime $p$.  The versality property may be easier to determine on the new varieties $X_p$ than on the original variety $X$.

This corollary is our main tool in the proof of Theorem \ref{thm:4surfacesClassification}.  In particular, we will show that the versality question on all toric surfaces can be reduced to studying $3$-groups acting on $\bbP^2$ and $2$-groups acting on $\bbP^1 \times \bbP^1$.

There does not seem to be any compelling reason why Corollary \ref{cor:versalByPGroups} should only be true for toric varieties since versality is a birational invariant.  One might conjecture that this theorem holds for \emph{any} variety:

\begin{conj}\label{conj:versalByPGroups}
Let $G$ be a finite group acting faithfully on a variety $X$.  Then $X$ is $G$-versal if and only if, for any prime $p$, $X$ is $G_p$-versal for a Sylow $p$-subgroup of $G$.
\end{conj}

\subsection{Monomial Actions}\label{sec:toricVersal:monomial}

We make the following observation:

\begin{lem}\label{lem:justTorus}
Suppose $X$ is a toric variety with a faithful action of a finite group $G$ contained in the torus $T$.  Then $X$ is $G$-versal.  Furthermore, if $X$ is complete, then $X$ has a $G$-fixed point.
\end{lem}

\begin{proof}
First, suppose $X$ is complete; by the Borel fixed point theorem $X$ has a $T$-fixed point and, thus, a $G$-fixed point.  In general, $X$ is $T$-equivariantly birationally equivalent to a complete non-singular toric variety (say $\bbP^n$).  Consequently, this birational equivalence is $G$-equivariant.  Thus $X$ is $G$-versal by Corollary \ref{cor:fixedPointImpliesVersal}.
\end{proof}

Consider a toric variety $X$ with torus $T$, fan $\Delta$ and lattice $N$.  Recall that $\Aut(N,\Delta)$ is the subgroup of $\GL(N)$ preserving the fan $\Delta$.  Note that the group $\Aut(N,\Delta)$ has a natural action on $X$ which is $T$-stable.  We say a group $G$ has a \emph{multiplicative action} on $X$ if $G \subset \Aut(N,\Delta)$.

\begin{lem}\label{lem:multAction}
Suppose $X$ is a toric variety with a faithful multiplicative action of a finite group $G$.  Then $X$ has a $G$-fixed point and $X$ is $G$-versal.
\end{lem}

\begin{proof}
Any element of $\Aut(N,\Delta)$ fixes the identity of the torus $T$ in $X$.  The versality of $X$ is well-known (see \cite[Lemma 3.3(d)]{ColliotTheleneKunyavskiiPopovReichstein2009Is-the-function} or \cite{Bannai2007Construction-of}).
\end{proof}

Note that both $T$-actions and multiplicative actions are $T$-stable --- they preserve $T$ as a subvariety of $X$.  Any particular $T$-stable automorphism of $X$ is a product of an element of $T$ and an element of $\Aut(N,\Delta)$.  Thus, the group of $T$-stable automorphisms of $X$ is precisely
\[ \Aut_T(X) = T \rtimes \Aut(N,\Delta) \,. \]
Given such a subgroup of $\Aut_T(X)$ there is a natural map
\[\omega_T: G \to \Aut(N,\Delta) \subset \GL(N)\]
given by the projection $G \to G/(G\cap T)$.  We denote this map $\omega_T$ to emphasize its dependence on $T$ (even though, strictly speaking, it depends on $N$).

Despite the fact that $T$-actions and multiplicative actions are always versal, this does not hold for $T$-stable actions in general.  Nevertheless, they are much more manageable than general actions.

\begin{defn}
Let $G$ be a group acting faithfully on a toric variety $X$.  We say that the action is \emph{monomial} if there exists a fan $\Delta$ in a lattice $N$ inducing a torus $T = N \otimes \bbC^\times$ such that the associated toric variety is $G$-equivariantly biregular to $X$ and $g(T) = T$ for all $g \in G$.
\end{defn}

Such actions are also called ``twisted multiplicative'' in the literature.  

Note that, for a linear variety $X \simeq \bbC^n$, monomial actions are precisely the same as monomial representations.  Recall that all linear representations of supersolvable groups are monomial \cite[Section 8.5, Theorem 16]{Serre1977Linear-represen}.  This result has a natural generalisation for toric varieties.

\begin{prop}\label{prop:supersolvableToric}
Suppose $G$ is a supersolvable finite group acting on a complete non-singular toric variety $X$.  Then $G$ is monomial.
\end{prop}

\begin{proof}
By Lemma \ref{lem:supersolvExtAreMonomial} below, there exists a change of basis $\alpha : V \to V$ such that $E = \pi^{-1}(G)$ has a monomial action on $\alpha(V)$ with $K$ acting diagonally.  Since $K$ acts diagonally in both coordinates, if $V_\lambda$ is the weight space in $V$ corresponding to some character $\lambda :K \to \bbC^\times$ then $\alpha(V_\lambda) \subset V_\lambda$.  These $V_\lambda$ are precisely the $S_{\alpha_i}'$ of \cite[Section 4]{Cox1995The-homogeneous}.  This means that $\alpha \in G_s$ where
\[ G_s = \prod \GL(S_{\alpha_i}') \subset \tildeAut(X).\]
Thus $\alpha$ descends to an automorphism of $X$.  In the new basis, we have an embedding $E \hookrightarrow (\bbC^\times)^{\Delta(1)} \rtimes \Aut(N,\Delta)$.  Taking the quotient by $K$ we obtain $G \subset T \rtimes \Aut(N,\Delta)$.
\end{proof}

\begin{lem}\label{lem:supersolvExtAreMonomial}
Suppose we have an exact sequence of algebraic groups (over $\bbC$)
\[ 1 \to K \to E \to G \to 1 \]
where $K$ is diagonalisable and $G$ is finite supersolvable.  For any representation $V$ of $E$ there exists a choice of coordinates such that $E$ is monomial with $K$ diagonal.  Furthermore, any irreducible representation has dimension dividing the order of $G$.
\end{lem}

\begin{proof}
This is a straight-forward generalisation of \cite[Section 8.5, Theorem 16]{Serre1977Linear-represen}.  We proceed by induction on the dimension of $V$.  For any normal subgroup $N \lhd E$ the quotient $\eta : E \to E/N$ sits in an exact sequence
\[ 1 \to \eta(K) \to \eta(E) \to \eta(E)/\eta(K) \to 1 \]
with $\eta(K)$ diagonalisable and $\eta(E)/\eta(K)$ finite supersolvable.  Thus it suffices to assume $V$ is a faithful irreducible representation of $E$.

Suppose $E$ is abelian.  There are no non-trivial unipotent elements in $G$ or $K$, so $E$ consists of semisimple elements.  Thus $E$ is diagonalisable (thus monomial).  This also takes care of the base case $\dim(V)=1$.

Suppose $E$ is non-abelian.  We claim there exists a normal diagonalisable subgroup $A$ containing $K$ which is not contained in the centre of $E$.  If $K$ is not central we may take $A = K$.  If $K$ is central then there exists a normal cyclic subgroup $C$ of $E/Z(E)$ by supersolvability of $E/K$.  In this case, take $A$ to be the inverse image of $C$ in $E$.  We see that $A$ is abelian (thus diagonalisable since $K$ is diagonalisable), contains $K$, and is not contained in the centre of $E$.

We have a decomposition $V=\oplus V_i$ into distinct weight spaces for the action of $A$.  Since $A$ is normal in $E$, the group $E$ permutes the spaces $V_i$.  In fact, the action of $E$ is transitive since $V$ is irreducible.  Since $E$ acts faithfully on $V$ and $A$ is not central in $E$, there is more than one weight space $V_i$.  Let $H$ be the maximal subgroup of $E$ such that $H(V_0)=V_0$.  We see that the $E$-representation $V$ is induced from the $H$-representation $V_0$.

Since $\dim(V) = [E:H] \dim(W)$ and $H$ contains $K$ the result follows from the induction hypothesis.
\end{proof}

Recall that $p$-groups are supersolvable.  Thus, in particular, actions of $p$-groups on toric varieties are always monomial.  This is particularly useful in light of Corollary \ref{cor:versalByPGroups} above.

\section{Del Pezzo Surfaces of Degree $\ge 5$}\label{sec:4surfaces}

The main goal of this section is to prove Theorem \ref{thm:4surfacesClassification}: a classification of all finite groups which act versally on one of the four surfaces $\bbP^2$, $\bbP^1 \times \bbP^1$, $DP_6$ (the del Pezzo surface of degree $6$) or $DP_5$ (the del Pezzo surface of degree $5$).

Recall that the automorphism group of $\bbP^2$ is $\PGL_3(\bbC)$; that of $\bbP^1 \times \bbP^1$ is
\[(\PGL_2(\bbC) \times \PGL_2(\bbC) ) \rtimes S_2 \]
where $S_2$ swaps the two copies of $\PGL_2(\bbC)$; that of $DP_6$ is $(\bbC^\times)^2 \rtimes D_{12}$ (see \cite[Section 6.2]{DolgachevIskovskikh2006Finite-subgroup}); and that of $DP_5$ is $S_5$ (see \cite[Section 6.3]{DolgachevIskovskikh2006Finite-subgroup}).

The surfaces $\bbP^2$, $\bbP^1 \times \bbP^1$ and $DP_6$ are toric.  The monomial actions on these surfaces will be particularly important.  For example, we have the following lemma:

\begin{lem}\label{lem:P1P1versalMonomial}
All versal actions of finite groups on $\bbP^1 \times \bbP^1$ are monomial.
\end{lem}

\begin{proof}
Recall that $\pi : \tildeAut(X) \to \Aut(X)$ is the group homomorphism induced from the Cox ring construction.  For $G \subset \Aut(X)$ we have the lift $E = \pi^{-1}(G)$ in
\[\tildeAut(X) \simeq (\GL_2(\bbC) \times \GL_2(\bbC) ) \rtimes S_2 \]
with the exact sequence (\ref{eqn:KEGseq}) from section \ref{sec:toricVersal}.

Let $H = G \cap (\PGL_2(\bbC) \times \PGL_2(\bbC) )$.  The group $H$ is the image in $G$ of the centralizer of $K$ in $E$.  We see that $H$ is a normal subgroup of $G$ of index at most $2$.  Let $H_1$ and $H_2$ be the projections of $H$ to the first and second copies of $\PGL_2(\bbC)$.  We note that there is a natural embedding $H \subset H_1 \times H_2$.  When $H \ne G$ we have isomorphisms $H_1 \simeq H_2$ induced by the actions of elements in $G - H$.

We may consider the action of $E$ on $V \simeq \bbC^4$ as a $4$-dimensional representation $\rho$.  Let $E_H = \pi^{-1}(H)$.  Note that $E_H \subset \GL_2(\bbC) \times \GL_2(\bbC)$.   Thus, the restriction $\rho|_{E_H}$ is a direct sum of $2$-dimensional subrepresentions $\sigma_1$ and $\sigma_2$.  Informally, one may consider $\sigma_1$ as the preimage of $H_1$ and $\sigma_2$ as the preimage of $H_2$.  If $G \ne H$ then $\rho$ is induced from $\sigma_1$.

If $G$ is versal then there is a section from $G$ to $E$.  Recall that a finite subgroup of $\PGL_2(\bbC)$ lifts isomorphically to $\GL_2(\bbC)$ if and only if it is cyclic or odd dihedral (Proposition \ref{prop:PGL2toGL2}).  All $2$-dimensional representations of lifts of such groups are monomial.  Thus $\sigma_1$ and $\sigma_2$ are monomial.  If $G=H$ then $\rho=\sigma_1 \oplus \sigma_2$ is monomial.  If $G \ne H$ then $\rho = \Ind_H^G \sigma_1$ is monomial.
\end{proof}

\subsection{Monomial Actions on Toric Surfaces}

Recall the classification of conjugacy classes of finite subgroups of $\GL_2(\bbZ)$.  We use the list in Lorenz's book \cite[page 30]{Lorenz2005Multiplicative-}.  The list comes with explicit representatives for each conjugacy class in terms of explicit matrix generators.  We use the $\calG_i$ notation to denote this explicit representative in each conjugacy class.  (Lorenz uses $\calG_i$ to denote the class, not the representative).  Since it is used so extensively in what follows, we reproduce the list in Table \ref{tab:conjClassGLZ}.

\begin{table}[ht]
\begin{center}
\begin{tabular}{|c|c|c|}
\hline
Label & Generators & Structure\\
\hline\hline
\spaceHack $\calG_1$ & $\mat{1&-1\\1&0}, \mat{0&1\\1&0}$ & $D_{12}$ \\ \hline
\spaceHack $\calG_2$ & $\mat{-1&0\\0&1}, \mat{0&1\\1&0}$ & $D_8$ \\ \hline
\spaceHack $\calG_3$ & $\mat{0&-1\\1&-1}, \mat{0&-1\\-1&0}$ & $D_6$ \\ \hline
\spaceHack $\calG_4$ & $\mat{0&-1\\1&-1}, \mat{0&1\\1&0}$ & $D_6$ \\ \hline
\spaceHack $\calG_5$ & $\mat{-1&0\\0&1}, \mat{1&0\\0&-1}$ & $C_2 \times C_2$ \\ \hline
\spaceHack $\calG_6$ & $\mat{0&1\\1&0}, \mat{0&-1\\-1&0}$ & $C_2 \times C_2$ \\ \hline
\spaceHack $\calG_7$ & $\mat{1&-1\\1&0}$ & $C_6$ \\ \hline
\spaceHack $\calG_8$ & $\mat{0&-1\\1&0}$ & $C_4$ \\ \hline
\spaceHack $\calG_9$ & $\mat{0&-1\\1&-1}$ & $C_3$ \\ \hline
\spaceHack $\calG_{10}$ & $\mat{-1&0\\0&-1}$ & $C_2$ \\ \hline
\spaceHack $\calG_{11}$ & $\mat{-1&0\\0&1}$ & $C_2$ \\ \hline
\spaceHack $\calG_{12}$ & $\mat{0&1\\1&0}$ & $C_2$ \\ \hline
\end{tabular}
\caption{Conjugacy classes of non-trivial finite subgroups of $\GL_2(\bbZ)$}
\label{tab:conjClassGLZ}
\end{center}
\end{table}

One checks that Figure \ref{fig:latticeGLZ} contains the finite subgroup lattice structure in $\GL_2(\bbZ)$ where an arrow means ``contains a subgroup in the conjugacy class of''.  We omit composite arrows for clarity.

\begin{figure}[ht]
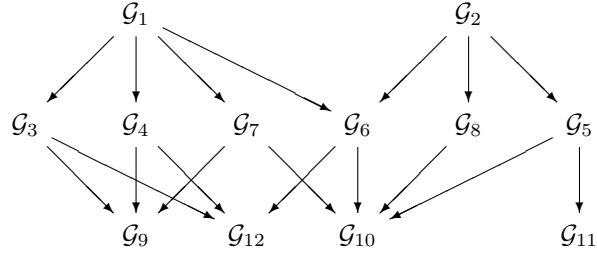

\[
\begin{diagram}
\node[2]{\calG_1} \arrow{sw} \arrow{s} \arrow{se} \arrow{see}
\node[3]{\calG_2} \arrow{sw} \arrow{s} \arrow{se}
\\
\node{\calG_3} \arrow{se} \arrow{see}
\node{\calG_4} \arrow{s} \arrow{se}
\node{\calG_7} \arrow{sw} \arrow{se}
\node{\calG_6} \arrow{sw} \arrow{s}
\node{\calG_8} \arrow{sw}
\node{\calG_5} \arrow{sww} \arrow{s}
\\
\node[2]{\calG_9}
\node{\calG_{12}}
\node{\calG_{10}}
\node[2]{\calG_{11}}
\end{diagram}
\]
\caption{Lattice of finite subgroups in $\GL_2(\bbZ)$}
\label{fig:latticeGLZ}
\end{figure}

From this subgroup lattice structure we make some useful observations about $p$-groups in $\GL_2(\bbZ)$.  For $p > 3$, there are no non-trivial $p$-subgroups of $\GL_2(\bbZ)$.  All $2$-subgroups of $\GL_2(\bbZ)$ are conjugate to a subgroup of $\calG_2$.  All non-trivial $3$-subgroups of $\GL_2(\bbZ)$ are conjugate to $\calG_9$.

Let $N = \bbZ^2$ be our lattice.  There are standard realisations of $\bbP^2$, $\bbP^1 \times \bbP^1$ and $DP_6$ as the toric varieties associated to the complete fans in $N$ from Figure \ref{fig:standardFans}.

\begin{figure}[ht]
\begin{center}
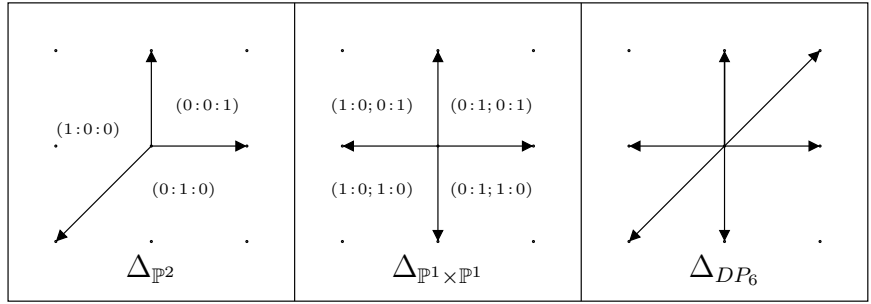
\end{center}
\caption{Standard fans for $\bbP^2$, $\bbP^1 \times \bbP^1$ and $DP_6$.}
\label{fig:standardFans}
\end{figure}

Let $T = N \otimes \bbC^\times$ be the torus associated to the lattice $N$.  Choose coordinates such that $(\lambda_1, \lambda_2) \in (\bbC^\times)^2 \simeq T$ corresponds to
\[ (\lambda_1:\lambda_2:1) \in \bbP^2 \]
and
\[ (\lambda_1:1\ ;\ \lambda_2:1) \in \bbP^1 \times \bbP^1. \]
Thus the maximal cones in $\Delta_{\bbP^2}$ and $\Delta_{\bbP^1 \times \bbP^1}$ correspond to the $T$-fixed points indicated in the diagram.  

Recall that the group of $T$-stable automorphisms $\Aut_T(X)$ of a surface $X$ is $T \rtimes \Aut(N,\Delta_X)$ where $\Aut(N,\Delta_X)$ is the group of automorphisms of the associated fan $\Delta_X$.  We have the following automorphism groups:
\begin{center}
\begin{tabular}{l c r}
$\Aut(N,\Delta_{\bbP^2}) = \calG_4$ &
$\Aut(N,\Delta_{\bbP^1 \times \bbP^1}) = \calG_2$ &
$\Aut(N,\Delta_{DP_6}) = \calG_1$
\end{tabular}
\end{center}

Since $\calG_1$ and $\calG_2$ are the maximal finite subgroups of $\GL_2(\bbZ)$ up to conjugacy (see Figure \ref{fig:latticeGLZ}), all monomial group actions on toric surfaces are equivariantly birational to actions on $\bbP^1 \times \bbP^1$ or $DP_6$.  (Note, however, that this is \emph{not} the case for general actions.)

By Lemma \ref{lem:justTorus}, this means that all faithful $p$-group actions on toric surfaces are automatically versal for $p > 5$.  For $3$-groups and $2$-groups the theory is a bit more involved.  For $3$-groups, the versal property can be determined by considering actions on $\bbP^2$:

\begin{lem}[$3$-groups acting on toric surfaces]\label{lem:versal3groups}
Suppose $G_3$ is a finite $3$-group acting faithfully on a toric surface $X$.  Then $X$ is $G_3$-birationally equivalent to $Y = \bbP^2$ with $G_3 \subset T \rtimes \calG_9$ and the following are equivalent:
\begin{enumerate}
\item[1.] $Y$ has a $G_3$-fixed point,
\item[2.] $X$ is $G_3$-versal,
\item[3.] the following conditions hold:
\begin{enumerate}
\item[(a)] if $\omega_T(G_3) = 1$ then there are no conditions,
\item[(b)] if $\omega_T(G_3) = \calG_9$ then $G_3 \cap T = 1$.
\end{enumerate}
\end{enumerate}
\end{lem}

\begin{proof}
Recall that all $p$-groups are monomial and $\calG_9$ is the maximal finite $3$-subgroup of $\GL_2(\bbZ)$ up to conjugacy.  Thus we may assume $G_3 \subset T \rtimes \calG_9$ in some coordinates by Proposition \ref{prop:supersolvableToric}.  Furthermore, selecting a new fan in the same lattice induces a birational map.  Thus $X$ is $G_3$-birationally equivalent to $Y=\bbP^2$.  Since versality is a $G$-birational invariant, it suffices to assume $X=Y$ for the remainder of the proof.

The implication $(1) \implies (2)$ is immediate by Corollary \ref{cor:fixedPointImpliesVersal}.  In case $(a)$, all remaining implications are immediate by Lemma \ref{lem:justTorus}.  It remains to consider case $(b)$ with $\omega_T(G_3) = \calG_9$.

Assume $(2)$: that $X$ is versal.  Since $G_3$ is a $3$-group of essential dimension $2$, it is abelian by Proposition \ref{prop:3groupsAbelian}.  Note that $\calG_9$ does not fix any of the cones of the fan $\Delta_{\bbP^2}$ except for the trivial cone $\{0\}$.  Thus any $G_3$-fixed point must be on the torus.  Note the action of any non-trivial element of $G_3 \cap T$ does not fix any point on the torus.  If $G_3 \cap T \ne 1$ then we have an abelian subgroup without a fixed point.  This contradicts Proposition \ref{prop:abelSubgroups}.  So $(3)$ must hold and we have $(2) \implies (3)$.

If $(3)$ holds, then $G_3 \simeq C_3$.  Any finite cyclic group acting on $\bbP^2$ has a diagonalisation.  So there exists a $G_3$-fixed point.  We have $(3) \implies (1)$.
\end{proof}

Similarly, we determine which $2$-groups are versal by studying $\bbP^1 \times \bbP^1$.

\begin{lem}[$2$-groups acting on toric surfaces]\label{lem:versal2groups}
Suppose $G_2$ is a finite $2$-group acting faithfully on a toric surface $X$.  Then $X$ is $G_2$-birationally equivalent to $Y = \bbP^1 \times \bbP^1$ with $G_2 \subset T \rtimes \calG_2$ and the following are equivalent:
\begin{enumerate}
\item[1.] $Y$ has a $G_2$-fixed point,
\item[2.] $X$ is $G_2$-versal,
\item[3.] the following conditions hold:
\begin{enumerate}
\item[(a)] if $\omega_T(G_2)$ is conjugate to $1$ or $\calG_{12}$ then there are no conditions,
\item[(b)] if $\omega_T(G_2)$ is conjugate to $\calG_{11}$ then (after choosing coordinates such that $\omega_T(G_2) = \calG_{11}$), for any $t \in G_2 \cap T$, $t = (1,\lambda)$ for some $\lambda \in \bbC^\times$,
\item[(c)] in all remaining cases we require $G_2 \cap T = 1$.
\end{enumerate}
\end{enumerate}
\end{lem}

\begin{proof}
Recall that $\calG_2$ is the maximal finite $2$-subgroup of $\GL_2(\bbZ)$ up to conjugacy.  Similarly to Lemma \ref{lem:versal3groups} above, we may assume $G_2 \subset T \rtimes \calG_2$ and $X=Y$.

The implication $(1) \implies (2)$ is immediate by Corollary \ref{cor:fixedPointImpliesVersal}.  We prove the remaining implications by restricting to each case in $(3)$.

\smallskip \noindent
{\bf Case $(a)$: $\omega_T(G_2)$ is conjugate to $1$ or $\calG_{12}$.}

There are no additional conditions so it suffices to show $(1)$ always holds.  When $\omega_T(G_2)=1$, this is immediate from Lemma \ref{lem:justTorus}.  For $\omega_T(G_2)$ conjugate to $\calG_{12}$ we choose coordinates so that $\omega_T(G_2)=\calG_{12}$ and use the fan $\Delta_{\bbP^1 \times \bbP^1}$ as above.  The cone $\sigma$ spanned by $\{(1,0),(0,1)\}$ is fixed by the action of $\calG_{12}$ and, since it is a maximal cone, corresponds to a $T$-fixed point.  Thus the $T$-orbit corresponding to $\sigma$ is a $G$-fixed point.

\smallskip \noindent
{\bf Case $(b)$: $\omega_T(G_2)$ is conjugate to $\calG_{11}$.}

It suffices to assume that $\omega_T(G_2) = \calG_{11}$ and $Y$ is constructed from the fan $\Delta_{\bbP^1 \times \bbP^1}$. Assume $(3)$ does not hold; we will show that this implies $(2)$ cannot hold.  There exists an element $t \in G_2 \cap T$ of the form $t = (\lambda_1,\lambda_2) \subset (\bbC^\times)^2$ where $\lambda_1 \ne 1$.  Furthermore, by taking appropriate powers, we may assume $\lambda_1$ has order $2$.  Now consider $g \in G_2$ such that $\omega_T(g) \ne 1$.  The group $A = \langle g, t \rangle$ is an abelian subgroup of $G_2$.

Note that $\calG_{11}$ (and thus $g$) only fixes the cones spanned by $\{(0,1)\}$, $\{(0,-1)\}$ and $\{0\}$ in $\Delta_{\bbP^1 \times \bbP^1}$.  The element $t$ acts non-trivially on the $T$-orbits corresponding to those cones (and so has no fixed points there).  We have an abelian subgroup $A$ without a fixed point.  This contradicts Proposition \ref{prop:abelSubgroups}.  Thus $(2)$ does not hold.  We have shown $(2) \implies (3)$.

Now assume $(3)$ holds.  Recall the definitions of $H$, $H_1$ and $H_2$ from the proof of Lemma \ref{lem:P1P1versalMonomial}.  In this case $G_2=H$; and $H_1$, $H_2$ are cyclic.  Thus $H_1$ has a fixed point $p_1$ on the first $\bbP^1$ and and $H_2$ has a fixed point $p_2$ on the second.  The point $(p_1,p_2)$ is a $G_2$-fixed point of $Y$.  Thus $(3) \implies (1) \implies (2)$.

\smallskip \noindent
{\bf Case $(c)$: all remaining cases.}

Assume $(3)$ does not hold.  Recall the subgroup structure of $\calG_2$ from figure \ref{fig:latticeGLZ}.  We must have $\calG_{10} \subset \omega_T(G_2)$.  If $G_2 \cap T \ne 1$ then there exists an element $t \in G_2 \cap T$ of order $2$.  The element $t$ commutes with the action of $\calG_{10}$.  Let $g \in G_2$ be an element such that $1 \ne \omega_T(g) \subset \calG_{10}$.  The group $A = \langle g, t \rangle$ is an abelian subgroup of $G_2$.

Note that $\calG_{10}$ only fixes the trivial cone $\{0\}$ in $\Delta_{\bbP^1 \times \bbP^1}$ so any $A$-fixed point must be on the torus.  The element $t$ does not fix any point on the torus.  We have an abelian subgroup $A$ without a fixed point.  This contradicts Proposition \ref{prop:abelSubgroups}.  Thus $(2)$ does not hold.  We have shown $(2) \implies (3)$.

Now assume $(3)$ holds.  In this case, $H_1$ and $H_2$ are cyclic.  A cyclic subgroup of $\PGL_2(\bbC)$ lies in some torus $\bbC^\times$.  Thus we may find new coordinates with a different torus $T' \subset Y$ such that $H \subset T'$.  Note that, for any $g \in G_2$, $g^2 \in T'$ so $\omega_{T'}(g)$ has order $2$.  Also, $\calG_{10}$ and $\calG_{11}$ are contained in $H$.  Thus, any $g \in G_2 - H$ must have $\omega_{T'}(g)$ conjugate to the non-trivial element in $\calG_{12}$.  So $G_2 \subset T' \rtimes \calG_{12}$ and has a fixed point by case $(a)$.  We have shown $(3) \implies (1) \implies (2)$.
\end{proof}

\begin{lem}\label{lem:versalMonomial}
All finite subgroups $G$ of the following groups have versal monomial actions on a toric surface:
\begin{enumerate}
\item[(\ref{thm:ed2:GL2}*)] $T \rtimes \calG_{12}$.
\item[(\ref{thm:ed2:G1})] $T \rtimes \calG_1$ with $|G \cap T|$ coprime to $2$ and $3$,
\item[(\ref{thm:ed2:G2})] $T \rtimes \calG_2$ with $|G \cap T|$ coprime to $2$,
\item[(\ref{thm:ed2:G3})] $T \rtimes \calG_3$ with $|G \cap T|$ coprime to $3$,
\item[(\ref{thm:ed2:G4})] $T \rtimes \calG_4$ with $|G \cap T|$ coprime to $3$,
\end{enumerate}
Furthermore, any finite group with a versal monomial action on a toric surface is of this form.
\end{lem}

\begin{proof}
Recall that when deciding whether a group $G$ has a versal action on a complete non-singular toric variety it suffices to check Sylow $p$-subgroups $G_p$ (Corollary \ref{cor:versalByPGroups}).  For all the forms above, $G_p$ is always versal when $p \ge 5$ by Lemma \ref{lem:justTorus}.  So one only needs to check the Sylow $3$- and $2$-subgroups.

Any finite $G$ with a monomial action can be written in the form $G \subset T \rtimes \calG_i$ where $\calG_i$ is from Table \ref{tab:conjClassGLZ}.  From Lemma \ref{lem:versal3groups} and Lemma \ref{lem:versal2groups} we have necessary and sufficient conditions for $G_2$ and $G_3$ to be versal.

We note $G_3 \cap T = 1$ is equivalent to $|G \cap T|$ coprime to $3$ and similarly for $G_2$.  By selecting appropriate Sylow subgroups we have Table \ref{tab:versalConditions} where the last row gives the necessary and sufficient conditions for $G$ to be versal.

\begin{table}[ht]
\begin{center}
\begin{tabular}{|c|c|c|c|}
\hline
$\omega_T(G)$ & $\omega_T(G_3)$ & $\omega_T(G_2)$ & $|G \cap T|$ coprime to\\
\hline\hline
\spaceHack $\calG_1$ & $\calG_9$ & $\calG_6$ & 2, 3 \\ \hline
\spaceHack $\calG_2$ & $1$ & $\calG_2$ & 2 \\ \hline
\spaceHack $\calG_3$ & $\calG_9$ & $\calG_{12}$ & 3 \\ \hline
\spaceHack $\calG_4$ & $\calG_9$ & $\calG_{12}$ & 3 \\ \hline
\spaceHack $\calG_5$ & $1$ & $\calG_5$ & 2 \\ \hline
\spaceHack $\calG_6$ & $1$ & $\calG_6$ & 2 \\ \hline
\spaceHack $\calG_7$ & $\calG_9$ & $\calG_{10}$ & 2, 3 \\ \hline
\spaceHack $\calG_8$ & $1$ & $\calG_8$ & 2 \\ \hline
\spaceHack $\calG_9$ & $\calG_9$ & $1$ & 3 \\ \hline
\spaceHack $\calG_{10}$ & $1$ & $\calG_{10}$ & 2 \\ \hline
\spaceHack $\calG_{11}$ & $1$ & $\calG_{11}$ & special \\ \hline
\spaceHack $\calG_{12}$ & $1$ & $\calG_{12}$ & none \\ \hline
\spaceHack $1$ & $1$ & $1$ & none \\ \hline
\end{tabular}
\caption{Versality conditions for monomial actions on surfaces.}
\label{tab:versalConditions}
\end{center}
\end{table}

One sees that any group $G$ listed in the theorem is versal.

For the converse, we show that all of the other possibilities for $\omega_T(G)$ are already contained in a group appearing in the list.  The cases $\calG_5$, $\calG_6$, $\calG_8$ and $\calG_{10}$ are all covered by form (\ref{thm:ed2:G2}); $\calG_7$, by form (\ref{thm:ed2:G1}); $\calG_9$, by forms (\ref{thm:ed2:G3}) and (\ref{thm:ed2:G4}); and $\omega_T(G)=1$ by form (\ref{thm:ed2:GL2}*).

It remains to eliminate the special case $\calG_{11}$.  Here, $G=H\subset H_1 \times H_2$ in the language of the proof of Lemma \ref{lem:P1P1versalMonomial}.  Thus any finite subgroup $G$ of $T \rtimes \calG_{11}$ must be a subgroup of $D_{2n} \times C_m$ for sufficiently large integers $n$ and $m$.  From case (3b) of Lemma \ref{lem:versal2groups}, if $G$ is versal we can assume $n$ is odd.  We show that any such group is actually isomorphic to a group of form (\ref{thm:ed2:GL2}*) above.

Indeed, consider the following subgroup of $\GL_2(\bbC)$:
\[ \left\langle \mat{\zeta_n&0\\0&\zeta_n^{-1}}, \mat{\zeta_m&0\\0&\zeta_m}, \mat{0&1\\1&0} \right\rangle \]
where $\zeta_n$ and $\zeta_m$ are $n$th and $m$th roots of unity, respectively.  This group is isomorphic to $D_{2n} \times C_m$ and has an embedding into $T \rtimes \calG_{12}$.
\end{proof}

\subsection{Versal Actions on the Four Surfaces}

\begin{thm}\label{thm:4surfacesClassification}
Suppose a finite group $G$ has a versal action on $\bbP^2$, $\bbP^1 \times \bbP^1$, $DP_6$, or $DP_5$.  Then $G$ is finite subgroup of one of the following groups:
\begin{enumerate}
\item[(\ref{thm:ed2:GL2})] $\GL_2(\bbC)$,
\item[(\ref{thm:ed2:G1})] $T \rtimes \calG_1$ with $|G \cap T|$ coprime to $2$ and $3$,
\item[(\ref{thm:ed2:G2})] $T \rtimes \calG_2$ with $|G \cap T|$ coprime to $2$,
\item[(\ref{thm:ed2:G3})] $T \rtimes \calG_3$ with $|G \cap T|$ coprime to $3$,
\item[(\ref{thm:ed2:G4})] $T \rtimes \calG_4$ with $|G \cap T|$ coprime to $3$,
\item[(\ref{thm:ed2:PSL27})] $\PSL_2(\bbF_7)$,
\item[(\ref{thm:ed2:S5})] $S_5$.
\end{enumerate}
Furthermore, all finite subgroups of the above groups act versally on one of those surfaces.
\end{thm}

\begin{proof}
Recall that any finite subgroup of $\GL_2(\bbC)$ acts versally on $\bbP^2$.  We note that any finite subgroup of $T \rtimes \calG_{12}$ is a subgroup of $\GL_2(\bbC)$.  So form (\ref{thm:ed2:GL2}*) in Lemma \ref{lem:versalMonomial} is wholly contained in form (\ref{thm:ed2:GL2}) of this theorem.

By Lemma \ref{lem:versalMonomial}, the versal monomial actions on \emph{any} toric surface are contained in forms (\ref{thm:ed2:GL2})--(\ref{thm:ed2:G4}) above.  Recall that the automorphism group of $DP_6$ is $T \rtimes \calG_1$ and the group of monomial automorphisms of $\bbP^1 \times \bbP^1$ is $T \rtimes \calG_2$.  Forms (\ref{thm:ed2:GL2})--(\ref{thm:ed2:G4}) all have versal actions on one of the $4$ surfaces.

It remains to study actions that are \emph{not} monomial.  All actions on $DP_6$ are monomial, and by Lemma \ref{lem:P1P1versalMonomial}, this is also true of versal actions on $\bbP^1 \times \bbP^1$.  Thus these surfaces require no more consideration.

We consider $DP_5$.  Recall from \cite[Section 6.2]{DolgachevIskovskikh2006Finite-subgroup} that a del Pezzo surface of degree $5$ can be described as a quotient $(\bbP^1)^5/\PSL_2(\bbC)$. The automorphism group of $DP_5$ is $S_5$ and its action is versal by the construction in \cite{BuhlerReichstein1997On-the-essentia}.  Thus, all subgroups of $S_5$ act versally on $DP_5$.

It remains to classify all finite groups acting versally on $\bbP^2$.  Recall that, by Corollary \ref{cor:versalPn}, it suffices to determine whether there is an isomorphic lift from $\PGL_3(\bbC)$ to $\GL_3(\bbC)$.  We refer to Blichfeldt's classification of finite subgroups of $\GL_3(\bbC)$ in \cite[Chapter V]{Blichfeldt1917Finite-Collinea}.  Using Blichfeldt's notation, we note that groups of types A and B descend to subgroups of $\GL_2(\bbC)$, and groups of type C and D descend to monomial actions on $\bbP^2$.  These groups have already been considered.

Finally, we consider the exceptions E--J in the classification.  Blichfeldt appends the symbol ``$\phi$'' to the order of a subgroup of $\GL_3(\bbC)$ when there is no isomorphic lift of its image in $\PGL_3(\bbC)$.  Consequently, only types H and J descend to versal actions --- these correspond to the groups $A_5$ and $\PSL_2(\bbF_7)$.
\end{proof}

\section{Conic Bundle Structures}\label{sec:conics}

Recall Manin and Iskovskikh's classification of minimal rational $G$-surfaces into conic bundles and del Pezzo surfaces from section \ref{sec:prelim:minimalModels}.  In this section, we establish the conic bundles case of Theorem \ref{thm:ed2to4surfaces}.  The del Pezzo surfaces case will be considered in section \ref{sec:DelPezzo}.

\begin{thm}\label{thm:conicVersal}
If $G$ has a versal action on a minimal conic bundle $X$ then $G$ has a versal action on $\bbP^1 \times \bbP^1$ or $\bbP^2$. 
\end{thm}

All of the following facts about conic bundle structures can be found in \cite[Section 5]{DolgachevIskovskikh2006Finite-subgroup} or \cite{Iskovskih1979Minimal-models-}.  A \emph{conic bundle structure} on a rational $G$-surface $X$ is a $G$-equivariant morphism $\phi : X \to B$ such that $B \simeq \bbP^1$ and the fibres are isomorphic to reduced conics in $\bbP^2$.  Note that, unlike del Pezzo surfaces, the $G$-action is required for this definition to make sense.  There may exist other group actions where $X$ does \emph{not} have such a structure (for example, not all actions on $\bbP^1 \times \bbP^1$ respect the fibration).

A fibre $F$ of the morphism $\phi$ is either isomorphic to $\bbP^1$ or to $\bbP^1 \wedge \bbP^1$ (two copies of $\bbP^1$ meeting at a point).  In the first case, $\Aut(F) \simeq \PGL_2(\bbC)$; in the second, $\Aut(F)$ has a monomial representation of degree $2$ (in particular, it is a subgroup of $\GL_2(\bbC)$).

Let $\mathbf{F}_n$ be the ruled surface $\bbP(\calO_{\bbP^1}\oplus \calO_{\bbP^1}(n))$ for a non-negative integer $n$ (see, for example, \cite[Section V.2]{Hartshorne1977Algebraic-geome}). A conic bundle is either isomorphic to some $\mathbf{F}_n$ or to a surface obtained from some $\mathbf{F}_n$ by blowing up a finite set of points, no two lying in a fibre of a ruling.

Let $\pi : G \to \PGL_2(\bbC)$ be the map induced by the action of $G$ on $B$ under $\phi$.  Let $G_B = \im(\pi)$ and $G_K = \ker(\pi)$.  One may consider $G_K$ as the largest subgroup of $G$ which preserves the generic fibre.  Note that every fibre of $\phi$ is $G_K$-invariant.  It is useful to think of $G_K$ as the group that ``acts on the fibre'' and $G_B$ as the group that ``acts on the base.''  Both $G_K$ and $G_B$ are polyhedral groups since they act faithfully on rational curves.

Let $\Sigma = \{p_1, \ldots, p_r \}$ be the set of points on $B$ whose fibres are singular.  Let $\calR$ be the set of components of singular fibres $\{ R_1 ,R_1', \ldots, R_r,R_r' \}$ where $R_i$ and $R_i'$ are the two components of the fibre $\phi^{-1}(p_i)$ for each $p_i \in \Sigma$.  We have a natural map $\xi : G \to \Aut(\calR)$ where $\Aut(\calR)$ is the group of permutations of $\calR$.  Let us denote $G_0 = \ker(\xi) \cap G_K$ (note that this differs slightly from the definition in \cite[Section 5.4]{DolgachevIskovskikh2006Finite-subgroup}).

\begin{proof}[Proof of Theorem \ref{thm:conicVersal}]
We prove the theorem by considering the different possibilities for $G_K$.  We suggest reviewing the results of section \ref{sec:prelim:polyhedral}.

Note that if $G_K$ contains a characteristic subgroup of order $2$ then $G$ has a non-trivial centre and, thus, a versal action on $\bbP^2$ by Proposition \ref{prop:centreMeansP2}.  The polyhedral groups with characteristic subgroups of order $2$ are the dihedral groups $D_{4n}$ with $n \ge 2$, and the cyclic groups of even order.  It remains to consider $G_K$ of the following types: odd cyclic, odd dihedral, $C_2 \times C_2$, $A_4$, $S_4$ and $A_5$.

By Lemma \ref{lem:G0faithfulOnFibres} below, $G_0$ acts faithfully on every component of every fibre of $\phi$.  If $S$ has no singular fibres then $X$ is a ruled surface and we may apply Lemma \ref{lem:conicRuled}.  Consequently, we may assume $\pi$ has a singular fibre $F$.  So $G_0$ acts faithfully on an irreducible component of $F$ with a fixed point.  Any such component is isomorphic to $\bbP^1$.  The only polyhedral groups with fixed points are the cyclic groups, so $G_0$ is odd cyclic.

Note that $G_K$ can only permute components of the same fibre, thus $\xi(G_K) \subset (C_2)^r$.  We have a normal structure with $G_0 \lhd G_K$ cyclic and $G_K/G_0 \subset (C_2)^r$.  This excludes $G_K \simeq A_4$, $G_K \simeq S_4$ and $G_K \simeq A_5$.  Thus it remains only to consider groups $G_K$ that are odd cyclic, odd dihedral or isomorphic to $C_2 \times C_2$.

These remaining cases are handled by the lemmas below.  If $G_K$ is odd cyclic then the result follows by Lemma \ref{lem:GKcyclic} below; this case corresponds to $X$ being a ruled surface.  If $G_K$ is odd dihedral then Lemma \ref{lem:GKdihedral} applies; these surfaces are the ``exceptional conic bundles'' of \cite[Section 5.2]{DolgachevIskovskikh2006Finite-subgroup}.  Finally, if $G_K \simeq C_2 \times C_2$ then Lemma \ref{lem:GKC2C2} applies; these are all ``non-exceptional conic bundles'' as in \cite[Section 5.4]{DolgachevIskovskikh2006Finite-subgroup}.
\end{proof}

\begin{lem}\label{lem:conicRuled}
If $X$ is a ruled surface with a versal $G$-action then $G$ acts versally on $\bbP^1 \times \bbP^1$ or $\bbP^2$.
\end{lem}

\begin{proof}
If $X$ is $\bbP^1 \times \bbP^1$ then we are done.  Otherwise, from \cite[Theorem 4.10]{DolgachevIskovskikh2006Finite-subgroup} we see that any finite group acting on a ruled surface is a central extension of a finite subgroup of $\PGL_2(\bbC)$ or $\SL_2(\bbC)$.  Any finite subgroup $G$ of such an extension has a versal action on $\bbP^2$.  Indeed, it suffices to consider $G$ with trivial centre.  Any such $G$ then embeds into $\PGL_2(\bbC)$ or $\SL_2(\bbC)$.  All polyhedral groups have versal actions on $\bbP^2$ (see proof of Theorem \ref{thm:4surfacesClassification}); as do all finite subgroups of $\SL_2(\bbC)$.
\end{proof}

\begin{lem}\label{lem:G0faithfulOnFibres}
The group $G_0$ acts faithfully on every component of every fibre of $\phi$.
\end{lem}

\begin{proof}
Let $R$ be a component of a fibre of $\phi$.  Since $G_0$ preserves components of fibres, we may $G_0$-equivariantly blowdown $X$ to a ruled surface such that $R$ is isomorphic to a fibre of the blowdown variety.  Thus, it suffices to prove the theorem for $X$ when all fibres are isomorphic to $\bbP^1$.

Let $g$ be any non-trivial element of $G_0$.  There exists an open cover of $B$ by open sets $U$ such that $\phi^{-1}(U) \simeq U \times \bbP^1$.  Let $V$ be the subset of distinct triples of points in $(\bbP^1)^3$.  There is an isomorphism $V \to \PGL_2(\bbC)$ by taking the automorphism determined by the images of the three points $0$, $1$ and $\infty$.  By composing this isomorphism with the restrictions $g|U\times \{0\}$, $g|U\times \{1\}$ and $g|U\times \{\infty\}$, we obtain a map $\gamma_{g,U} : U \to \PGL_2(\bbC)$ which takes each point to the action of $g$ on the fibre of $\phi$.

Let $\alpha : \PGL_2(\bbC) \to \bbC$ be the map defined by
\[ \alpha : A \mapsto \frac{\Tr(A')^2}{\det(A')} \]
where $A'$ is any lift of $A$ to $\GL_2(\bbC)$.  One easily checks that $\alpha$ is well-defined and is invariant on conjugacy classes.  Furthermore, for any $A \in \PGL_2(\bbC)$ of finite order, $\alpha(A)=4$ if and only if $A=1$ (by diagonalisation).

The isomorphism $\phi^{-1}(U) \simeq U \times \bbP^1$ is only determined up to conjugacy in $\PGL_2(\bbC)$.  Gluing together each $\gamma_{g,U}$ after composing with $\alpha$ we obtain a map $\gamma_g : B \to \bbC$.  Since $\bbC$ is affine and $B \simeq \bbP^1$, the image of $\gamma_g$ is a point.  Since $G_0$ acts faithfully on $X$, there must be at least one fibre on which $g$ acts non-trivially.  Thus $\gamma_g \ne 4$ and $g$ acts non-trivially on every fibre.  Thus $G_0$ acts faithfully on every fibre.
\end{proof}

\begin{lem}\label{lem:GKcyclic}
Suppose $G$ acts versally on $X$ and $G_K$ is odd cyclic.  Then $G$ acts versally on $\bbP^1 \times \bbP^1$ or $\bbP^2$.
\end{lem}

\begin{proof}
It suffices to consider $G$-minimal $X$.  When $X$ is a ruled surface then the result follows from Lemma \ref{lem:conicRuled}.  As we shall see, this is the only case that occurs.

Suppose $X$ is \emph{not} a ruled surface.  We will show that $G_K$ must contain an involution, contradicting the assumption that $G_K$ has odd order.  We use the same reasoning as the proof of \cite[Lemma 5.6]{DolgachevIskovskikh2006Finite-subgroup}.

Since $X$ is $G$-minimal, there must exist an element $g \in G$ that swaps two components, $R$ and $R'$, of a singular fibre of $\phi$.  By taking an odd power, we may assume that $g$ has order $m=2^a$.

Consider $a=1$.  The intersection point $p$ of $R$ and $R'$ is in the fixed locus $X^g$.  Any involution acting on a surface with an isolated fixed point must act via $(x,y)\mapsto(-x,-y)$ in some local coordinates about that point. Thus, if $p$ is an isolated fixed point then $g$ cannot swap $R$ and $R'$.  This contradiction insures that $X^g$ contains a curve other than the fibres of $\phi$.  Thus, $g$ is contained in $G_K$.

Now, consider the remaining case $a>1$.  Consider $h = g^{m/2}$.  Suppose $X^h$ contains $R$.  Then $hg(y)=gh(y)=g(y)$ applies for all $y \in R$.  This means that $R'$ is contained in $X^h$ as well, contradicting the smoothness of $X^h$.  Thus, neither component is contained in $X^h$.

There exists exactly one fixed point $y$ on $R$ other than its intersection with $R'$.  If $y$ was an isolated $h$-fixed point on $X$ then its image $q$ would still be an isolated $h$-fixed point upon blowing down $R$.  But then $R$ has a trivial $h$-action: a contradiction.  Thus $h$ fixes a curve not contained in the fibres of $\phi$.  We obtain $h \in G_K$.
\end{proof}

\begin{lem}\label{lem:GKdihedral}
Suppose $G$ acts versally on $X$ and $G_K$ is odd dihedral.  Then $G$ acts versally on $\bbP^1 \times \bbP^1$.
\end{lem}

\begin{proof}
Recall that $G_0 \simeq C_n$ and $G_K \simeq D_{2n}$ for some $n$ odd as in the proof of Theorem \ref{thm:conicVersal}.  Consider any $g \in G_B$, we shall prove that $H = \pi^{-1}(\langle g \rangle)$ is a direct product $G_K \times \langle g \rangle$.

Since $\langle g \rangle$ is cyclic there is a fixed point on $B$.  Thus $\phi$ has an $H$-fixed fibre $F$.  Recall that $G_0$ acts faithfully on $F$.  If $F$ is non-singular then $\Aut(F) \simeq \PGL_2(\bbC)$.  If $F$ is singular then $\Aut(F) \subset \GL_2(\bbC)$ and we have a natural map $\Aut(F) \to \PGL_2(\bbC)$ with central kernel.  The group $G_0$ is not in the centre of $G_K$, so we have map $\eta : H \to \PGL_2(\bbC)$ which is injective on $G_0$.

Since $\eta$ is injective on $G_0$ it must be injective on $G_K$.  The image of $\eta$ must be a polyhedral group with a normal subgroup isomorphic to $G_K \simeq D_{2n}$ for some odd $n$.  From Lemma \ref{prop:polyNormalStructure}, the only possibilities are $\eta(H) \simeq G_K$ or $\eta(H) \simeq D_{4n} \simeq G_K \times C_2$ (since $n$ is odd).  Either way, there exists a retract $H \to G_K$ of the inclusion $G_K \hookrightarrow H$.  Thus $H \simeq G_K \times \langle g \rangle$.

Recall that $G_K$ has a trivial centre.  By \cite[Corollary IV.6.8]{Brown1982Cohomology-of-g}, there is only one extension of $G_K$ by $G_B$ associated to a map $G_B \to \Out(G_K)$ (up to equivalence).  Since $g$ has a trivial action on $G_K$ for any $g \in G_B$, the map $G_B \to \Out(G_K)$ is trivial.  Thus we must have $G \simeq G_K \times G_B$.

The group $G_K$ contains an involution.  If $G_B$ contains a subgroup isomorphic to $C_2 \times C_2$ then $G$ contains $(C_2)^3$.  This would contradict $\ed(G) \le 2$.  So $G_B$ must be cyclic or odd dihedral.  Thus, $G \simeq G_K \times G_B$ has a versal action on $\bbP^1 \times \bbP^1$ where $G_K$ acts on one $\bbP^1$, and $G_B$, the other.
\end{proof}

\begin{lem}\label{lem:GKC2C2}
Suppose $G$ acts versally on $X$ and $G_K \simeq C_2 \times C_2$.  Then $G$ acts versally on $\bbP^2$.
\end{lem}

\begin{proof}
It suffices to consider $G$ with a trivial centre, since otherwise we immediately have a versal action on $\bbP^2$ by Proposition \ref{prop:centreMeansP2}.  We have a map $G_B \to \Aut(G_K) \simeq S_3$ with kernel $J$.  We note that, by construction, if $g \in G$ maps to $J \subset G/G_K$ then $g$ commutes with $G_K$.

Suppose $J = 1$.  If $G \to S_3$ is not surjective then $G$ is abelian or isomorphic to $A_4$.  Both of these have versal actions on $\bbP^2$ so it suffices to assume $G$ is an extension of $C_2^2$ by $S_3$.  A $2$-Sylow subgroup of $G$ is not normal, since we would obtain a non-trivial map from $S_3$ to $C_3$ (which cannot exist).  A $3$-Sylow subgroup of $G$ is not normal since $A_4 \subset G$ and $C_3$ is not normal in $A_4$.  The only group $G$ of this form is $S_4$ \cite[Theorem 1.33]{Isaacs2008Finite_group_th}.  The group $S_4$ has a versal action on $\bbP^2$ by the proof of Theorem \ref{thm:4surfacesClassification}.

It remains to consider $J \ne 1$.  We shall see that this case cannot occur.

Suppose $J$ contains a subgroup $M \simeq C_2 \times C_2$.  The group $M'=\pi^{-1}(M) \subset G$ has essential dimension $\le 2$ and a non-trivial centre (it is a $2$-group).  Thus, there is an embedding $\rho : M' \hookrightarrow \GL_2(\bbC)$.  This representation $\rho$ is faithful and, since $G_K \subset Z(M')$, has a non-cyclic centre.  It cannot be irreducible by Schur's lemma.  Thus, $M'$ is abelian and must have a fixed point on $X$ (since it is versal).

Under the projection to $B$ this becomes a fixed point for $M$.  But $M$ has rank $2$ and cannot have a fixed point on $B \simeq \bbP^1$, a contradiction.  Thus we cannot have a subgroup $C_2 \times C_2$ in $J$.  We have a morphism $G_B \to S_3$ whose kernel cannot contain $C_2 \times C_2$.  Considering the normal structure of $G_B$ (a polyhedral group), this excludes $G_B$ isomorphic to $A_4$, $S_4$ or $A_5$.

It remains to consider $G_B$ cyclic or dihedral.  The involutions in $\Aut(G_K)$ all fix a non-trivial element of $G_K$.  Since $G$ has a trivial centre, we must have an element $g \in G$ that descends to an element of order $3$ in $\Aut(G_K)$.

If $G_B$ is cyclic then $J$ and $\pi(g)$ generate $G_B$.  If $G_B$ is dihedral then $J$ and $\pi(g)$ generate the maximal normal cyclic subgroup of $G_B$.  Indeed, there is no non-trivial map from a dihedral group to $C_3$ so $G_B$ surjects onto $\Aut(G_K) \simeq S_3$.  The kernel of the composition $G_B \to \Aut(G_K) \to C_2$ is generated by $J$ and $\pi(g)$ as desired.  Note that $\pi(g)^3 \in J$ in either case.

Let $L = \langle \pi^{-1}(J), g \rangle$.  Consider any $j \in \pi^{-1}(J)$.  Since $\pi(j)$ and $\pi(g)$ commute, we have $(g,j)=k$ for some $k \in G_K$.  Thus $j^2 \in Z(L)$ since $gj^2g^{-1}=k^2j^2=j^2$.

Suppose $J$ has even order.  Note that $Z(L) \cap G_K = 1$ so there exists $j \in \pi^{-1}(J)$ with $j \notin G_K$ such that $j^2=1$.  In this case, we have a subgroup $G_K \times \langle j \rangle \simeq (C_2)^3 \subset G$.  This cannot have essential dimension $2$ so we have a contradiction.

We may assume $J$ has odd order.  Note that $Z(L) \subset \pi^{-1}(J)$ since any element mapping non-trivially to $L/\pi^{-1}(J) \subset \Aut(G_K)$ cannot be central.  We want to show that $\pi$ maps $Z(L)$ onto $J$.  For any $y \in J$ there exists $x \in J$ such that $x^2 = y$ (since $|J|$ is odd).  There is a lift $l \in L$ such that $\pi(l) = x$.  We have $l^2 \in Z(L)$ and $\pi(l^2) = x^2=y$ as desired.  Since $G_K \cap Z(L) = 1$ we have a splitting $J \hookrightarrow L$ with image $Z(L)$.  Thus we may identify $J$ and $Z(L)$

Since $\ed(L) \le 2$ and $J = Z(L) \ne 1$, there is an embedding $L \hookrightarrow \GL_2(\bbC)$.  We then compose this with the natural map $\GL_2(\bbC) \to \PGL_2(\bbC)$.  Note that $L/J \simeq (G_K \rtimes C_3) \simeq A_4$.  Since $Z(L) = J$, we have a map $L \to \PGL_2(\bbC)$ with image $A_4$ and kernel $J$.  Any subgroup of $\GL_2(\bbC)$ mapping onto $A_4 \subset \PGL_2(\bbC)$ must have a central involution by Proposition \ref{prop:PGL2toGL2}.  So $J$ has even order; a contradiction.
\end{proof}

\section{Del Pezzo Surfaces of Degree $\le 4$}\label{sec:DelPezzo}

We are finally in a position to prove Theorem \ref{thm:ed2classification}.  It remains only to show that groups with versal actions on del Pezzo surface of degree $\le 4$ have already been seen acting versally on the surfaces of Theorem \ref{thm:4surfacesClassification}.  Indeed, the main theorem is an immediate consequence of the following:

\newtheorem*{ed2to4surfacesHack}{Theorem \ref{thm:ed2to4surfaces}}
\begin{ed2to4surfacesHack}
If $G$ is a finite group of essential dimension $2$ then $G$ has a versal action on $\bbP^2$, $\bbP^1 \times \bbP^1$, $DP_6$, or $DP_5$.
\end{ed2to4surfacesHack}

\begin{proof}
All groups $G$ of essential dimension $2$ have versal actions on minimal rational $G$-surfaces by Proposition \ref{prop:ratMinModelStrategy}.  Thus, it suffices to prove that, for any minimal rational versal $G$-surface $X$, there exists a versal action on one of the $4$ surfaces listed above.  Recall that any minimal rational $G$-surface $X$ is a del Pezzo surface or has a conic bundle structure by Theorem \ref{thm:EMI}.

Theorem \ref{thm:conicVersal} proves the theorem for surfaces with a conic bundle structure.  We recall from \cite[Section 6]{DolgachevIskovskikh2006Finite-subgroup} that the only minimal rational $G$-surfaces of degree $\ge 5$ are precisely those listed in the statement of the theorem.  Thus it suffices to consider degrees $\le 4$.  In the following $X$ is a del Pezzo surface with a versal $G$-action.

\smallskip \noindent
{\bf Case degree $4$:}
The minimal groups of automorphisms of del Pezzo surfaces $X$ of degree $4$ are listed in \cite[Theorem 6.9]{DolgachevIskovskikh2006Finite-subgroup}.  We know that $G$ must be from this list and that $\ed(G) \le 2$.  If $G$ is abelian or a $2$-group then it acts versally on $\bbP^2$.  All remaining groups have abelian subgroups with ranks $\ge 3$ (note that $C_2 \times A_4$ contains $C_2^3$); thus they cannot be versal.

An alternative proof that does not rely directly on \cite[Theorem 6.9]{DolgachevIskovskikh2006Finite-subgroup} can be found in the appendix of \cite{Duncan2009Finite_Groups_o}.

\smallskip \noindent
{\bf Case degree $3$:}
The minimal groups of automorphisms of del Pezzo surfaces $X$ of degree $3$ are listed in \cite[Theorem 6.14]{DolgachevIskovskikh2006Finite-subgroup}.  It suffices to consider $G$ from this list.

All groups with non-trivial centres and essential dimension $\le 2$ have versal actions on $\bbP^2$ by Proposition \ref{prop:centreMeansP2}.  Thus we may assume $G$ has a trivial centre.  In particular, we may eliminate all abelian groups from the list.  Next, we may eliminate all groups with abelian subgroups of rank $\ge 3$ since they cannot be versal by Proposition \ref{prop:edResults}(\ref{prop:edResults:abelian}).  Similarly, we eliminate $G$ containing a non-abelian $3$-subgroup by Proposition \ref{prop:3groupsAbelian}.  Also, if $G$ is a subgroup of $S_5$ then $G$ has a versal action on $DP_5$ by the proof of Theorem \ref{thm:4surfacesClassification}.

All that remains to consider are $G$ of the form $C_3^2 \rtimes C_2$ and $C_3^2 \rtimes C_2^2$.  It suffices to consider $G \simeq C_3^2 \rtimes C_2^2$.  We may view this group as a representation of $C_2^2$ on the vector space $\bbF_3^2$.  Since the centre is trivial, we may assume the representation is faithful.  The representation is diagonalisable, so $G$ is isomorphic to $S_3 \times S_3$.  This group has a versal action on $\bbP^1 \times \bbP^1$ by the proof of Theorem \ref{thm:4surfacesClassification}.

An alternative proof can be found in the appendix of \cite{Duncan2009Finite_Groups_o}.

\smallskip \noindent
{\bf Case degree $2$:}
We have a finite $G$-equivariant morphism of degree $2$ to $\bbP^2$ (see \cite[Section 6.6]{DolgachevIskovskikh2006Finite-subgroup}).  If the induced action of $G$ on $\bbP^2$ is faithful then we are done.  Otherwise, the group $G$ contains a central involution (a \emph{Geiser involution}).  Any such group has a non-trivial centre and sits inside $\GL_2(\bbC)$ by Corollary \ref{prop:centreMeansP2}.  

\smallskip \noindent
{\bf Case degree $1$:}
This case proceeds the same way as degree $2$ via the \emph{Bertini involution}.  The only difference is that the finite morphism of degree $2$ maps onto a singular quadric cone in $\bbP^3$ (\cite[Section 6.7]{DolgachevIskovskikh2006Finite-subgroup}).  The automorphism group of a singular quadric cone is the same as the minimal ruled surface $\bfF_2$ (see \cite[Example V.2.11.4]{Hartshorne1977Algebraic-geome}).  Any versal action on such a surface must also act versally on $\bbP^2$ or $\bbP^1 \times \bbP^1$ (see Lemma \ref{lem:conicRuled} above).
\end{proof}

\subsection*{Acknowledgements}
The author was partially supported by a Canada Graduate Scholarship from the Natural Sciences and Engineering Research Council of Canada.  The author would like to thank Z. Reichstein for providing extensive advice and comments on several early drafts of this paper.  The author also thanks R. L\"otscher for useful discussions, and J-P. Serre and the referee for some helpful comments.

\addcontentsline{toc}{section}{References}
\bibliographystyle{plain}
\bibliography{bibliography}

\begin{thebibliography}{10}

\bibitem{Bannai2007Construction-of}
Shinzo Bannai.
\newblock Construction of versal {G}alois coverings using toric varieties.
\newblock {\em Osaka J. Math.}, 44(1):139--146, 2007.

\bibitem{BannaiTokunaga2007A-note-on-embed}
Shinzo Bannai and Hiro-o Tokunaga.
\newblock A note on embeddings of {$S\sb 4$} and {$A\sb 5$} into the
  two-dimensional {C}remona group and versal {G}alois covers.
\newblock {\em Publ. Res. Inst. Math. Sci.}, 43(4):1111--1123, 2007.

\bibitem{BayleBeauville2000Birational_invo}
Lionel Bayle and Arnaud Beauville.
\newblock Birational involutions of {${\bf P}\sp 2$}.
\newblock {\em Asian J. Math.}, 4(1):11--17, 2000.
\newblock Kodaira's issue.

\bibitem{Beauville2007p-elementary-su}
Arnaud Beauville.
\newblock {$p$}-elementary subgroups of the {C}remona group.
\newblock {\em J. Algebra}, 314(2):553--564, 2007.

\bibitem{BerhuyFavi2003Essential-dimen}
Gr{\'e}gory Berhuy and Giordano Favi.
\newblock Essential dimension: a functorial point of view (after {A}.
  {M}erkurjev).
\newblock {\em Doc. Math.}, 8:279--330 (electronic), 2003.

\bibitem{Blanc2007Finite-abelian-}
J{\'e}r{\'e}my Blanc.
\newblock Finite abelian subgroups of the {C}remona group of the plane.
\newblock {\em C. R. Math. Acad. Sci. Paris}, 344(1):21--26, 2007.

\bibitem{Blichfeldt1917Finite-Collinea}
Hans~F. Blichfeldt.
\newblock {\em Finite Collineation Groups}.
\newblock The University of Chicago Press, 1917.

\bibitem{BrosnanReichsteinVistoli2007Essential-dimen}
Patrick Brosnan, Zinovy Reichstein, and Angelo Vistoli.
\newblock Essential dimension and algebraic stacks.
\newblock arXiv:math/0701903v1, 2007.

\bibitem{Brown1982Cohomology-of-g}
Kenneth~S. Brown.
\newblock {\em Cohomology of groups}, volume~87 of {\em Graduate Texts in
  Mathematics}.
\newblock Springer-Verlag, New York, 1982.

\bibitem{BuhlerReichstein1997On-the-essentia}
Joe Buhler and Zinovy Reichstein.
\newblock On the essential dimension of a finite group.
\newblock {\em Compositio Math.}, 106(2):159--179, 1997.

\bibitem{ChuHuKangZhang2008Groups_with_ess}
Huah Chu, Shou-Jen Hu, Ming-Chang Kang, and Jiping Zhang.
\newblock Groups with essential dimension one.
\newblock {\em Asian J. Math.}, 12(2):177--191, 2008.

\bibitem{ColliotTheleneKunyavskiiPopovReichstein2009Is-the-function}
Jean-Louis Colliot-Th\'el\`ene, Boris Kunyavski\v{\i}, Vladimir~L. Popov, and
  Zinovy Reichstein.
\newblock Is the function field of a reductive {L}ie algebra purely
  transcendental over the field of invariants for the adjoint action?
\newblock arXiv:0901.4358v3 [math.AG], 2009.

\bibitem{ColliotTheleneSansuc1987La_descente_sur}
Jean-Louis Colliot-Th{\'e}l{\`e}ne and Jean-Jacques Sansuc.
\newblock La descente sur les vari\'et\'es rationnelles. {II}.
\newblock {\em Duke Math. J.}, 54(2):375--492, 1987.

\bibitem{Coray1987Cubic_hypersurf}
Daniel~F. Coray.
\newblock Cubic hypersurfaces and a result of {H}ermite.
\newblock {\em Duke Math. J.}, 54(2):657--670, 1987.

\bibitem{Cox1995The-homogeneous}
David~A. Cox.
\newblock The homogeneous coordinate ring of a toric variety.
\newblock {\em J. Algebraic Geom.}, 4(1):17--50, 1995.

\bibitem{Coxeter1991Regular-complex}
Harold S.~M. Coxeter.
\newblock {\em Regular complex polytopes}.
\newblock Cambridge University Press, Cambridge, second edition, 1991.

\bibitem{deFernex2004On_planar_Cremo}
Tommaso de~Fernex.
\newblock On planar {C}remona maps of prime order.
\newblock {\em Nagoya Math. J.}, 174:1--28, 2004.

\bibitem{DeMeyer1983Generic_polynom}
Frank~R. DeMeyer.
\newblock Generic polynomials.
\newblock {\em J. Algebra}, 84(2):441--448, 1983.

\bibitem{Derenthal2007Universal_torso}
Ulrich Derenthal.
\newblock Universal torsors of del {P}ezzo surfaces and homogeneous spaces.
\newblock {\em Adv. Math.}, 213(2):849--864, 2007.

\bibitem{DolgachevIskovskikh2006Finite-subgroup}
Igor~V. Dolgachev and Vasily~A. Iskovskikh.
\newblock Finite subgroups of the plane {C}remona group, 2006.
\newblock arXiv:math/0610595v4 [math.AG].

\bibitem{Duncan2009Finite_Groups_o}
Alexander Duncan.
\newblock Finite groups of essential dimension 2, 2009.
\newblock arXiv:0912.1644v1 [math.AG].

\bibitem{Duncan2010Essential-dimen}
Alexander Duncan.
\newblock Essential dimensions of {$A_7$} and {$S_7$}.
\newblock {\em Math. Res. Lett.}, 17(2):263--266, 2010.

\bibitem{Fulton1993Introduction-to}
William Fulton.
\newblock {\em Introduction to toric varieties}, volume 131 of {\em Annals of
  Mathematics Studies}.
\newblock Princeton University Press, Princeton, NJ, 1993.
\newblock , The William H. Roever Lectures in Geometry.

\bibitem{GaribaldiMerkurjevSerre2003Cohomological-i}
Skip Garibaldi, Alexander Merkurjev, and Jean-Pierre Serre.
\newblock {\em Cohomological invariants in {G}alois cohomology}, volume~28 of
  {\em University Lecture Series}.
\newblock American Mathematical Society, Providence, RI, 2003.

\bibitem{Hartshorne1977Algebraic-geome}
Robin Hartshorne.
\newblock {\em Algebraic geometry}.
\newblock Springer-Verlag, New York, 1977.
\newblock Graduate Texts in Mathematics, No. 52.

\bibitem{HuKeel2000Mori_dream_spac}
Yi~Hu and Sean Keel.
\newblock Mori dream spaces and {GIT}.
\newblock {\em Michigan Math. J.}, 48:331--348, 2000.
\newblock Dedicated to William Fulton on the occasion of his 60th birthday.

\bibitem{Isaacs2008Finite_group_th}
I.~Martin Isaacs.
\newblock {\em Finite group theory}, volume~92 of {\em Graduate Studies in
  Mathematics}.
\newblock American Mathematical Society, Providence, RI, 2008.

\bibitem{Iskovskih1979Minimal-models-}
Vasily~A. Iskovskikh.
\newblock Minimal models of rational surfaces over arbitrary fields.
\newblock {\em Izv. Akad. Nauk SSSR Ser. Mat.}, 43(1):19--43, 237, 1979.
\newblock English translation: Math. USSR-Izv. 14 (1980), no. 1, 17--39.

\bibitem{JensenLedetYui2002Generic-polynom}
Christian~U. Jensen, Arne Ledet, and Noriko Yui.
\newblock {\em Generic polynomials}, volume~45 of {\em Mathematical Sciences
  Research Institute Publications}.
\newblock Cambridge University Press, Cambridge, 2002.
\newblock Constructive aspects of the inverse Galois problem.

\bibitem{KarpenkoMerkurjev2008Essential-dimen}
Nikita~A. Karpenko and Alexander~S. Merkurjev.
\newblock Essential dimension of finite {$p$}-groups.
\newblock {\em Invent. Math.}, 172(3):491--508, 2008.

\bibitem{KemperMattig2000Generic-polynom}
Gregor Kemper and Elena Mattig.
\newblock Generic polynomials with few parameters.
\newblock {\em J. Symbolic Comput.}, 30(6):843--857, 2000.
\newblock Algorithmic methods in Galois theory.

\bibitem{Kraft2006A_result_of_Her}
Hanspeter Kraft.
\newblock A result of {H}ermite and equations of degree 5 and 6.
\newblock {\em J. Algebra}, 297(1):234--253, 2006.

\bibitem{KraftLotscherSchwarz2009Compression-of-}
Hanspeter Kraft, Roland L{\"o}tscher, and Gerald~W. Schwarz.
\newblock Compression of finite group actions and covariant dimension. {II}.
\newblock {\em J. Algebra}, 322(1):94--107, 2009.

\bibitem{KraftSchwarz2007Compression-of-}
Hanspeter Kraft and Gerald~W. Schwarz.
\newblock Compression of finite group actions and covariant dimension.
\newblock {\em J. Algebra}, 313(1):268--291, 2007.

\bibitem{LafaceVelasco2009A-survey-on-Cox}
Antonio Laface and Mauricio Velasco.
\newblock A survey on {C}ox rings.
\newblock {\em Geom. Dedicata}, 139:269--287, 2009.

\bibitem{Ledet2007Finite-groups-o}
Arne Ledet.
\newblock Finite groups of essential dimension one.
\newblock {\em J. Algebra}, 311(1):31--37, 2007.

\bibitem{Lorenz2005Multiplicative-}
Martin Lorenz.
\newblock {\em Multiplicative invariant theory}, volume 135 of {\em
  Encyclopaedia of Mathematical Sciences}.
\newblock Springer-Verlag, Berlin, 2005.
\newblock Invariant Theory and Algebraic Transformation Groups, VI.

\bibitem{Lotscher2008Application-of-}
Roland L\"otscher.
\newblock Application of multihomogeneous covariants to the essential dimension
  of finite groups.
\newblock arXiv:0811.3852v1 [math.AG], 2008.

\bibitem{Manin1967Rational-surfac}
Yuri~I. Manin.
\newblock Rational surfaces over perfect fields. {II}.
\newblock {\em Mat. Sb. (N.S.)}, 72 (114):161--192, 1967.

\bibitem{Manin1986Cubic-forms}
Yuri~I. Manin.
\newblock {\em Cubic forms}, volume~4 of {\em North-Holland Mathematical
  Library}.
\newblock North-Holland Publishing Co., Amsterdam, second edition, 1986.
\newblock Algebra, geometry, arithmetic, Translated from the Russian by M.
  Hazewinkel.

\bibitem{MeyerReichstein2008Some_consequenc}
Aurel Meyer and Zinovy Reichstein.
\newblock Some consequences of the {K}arpenko-{M}erkurjev theorem.
\newblock arXiv:0811.2517v1 [math.AG], 2008.

\bibitem{Prokhorov2009Simple-finite-s}
Yuri Prokhorov.
\newblock Simple finite subgroups of the {C}remona group of rank 3.
\newblock arXiv:0908.0678v3 [math.AG], 2009.

\bibitem{Reichstein2000On-the-notion-o}
Zinovy Reichstein.
\newblock On the notion of essential dimension for algebraic groups.
\newblock {\em Transform. Groups}, 5(3):265--304, 2000.

\bibitem{Reichstein2004Compressions-of}
Zinovy Reichstein.
\newblock Compressions of group actions.
\newblock In {\em Invariant theory in all characteristics}, volume~35 of {\em
  CRM Proc. Lecture Notes}, pages 199--202. Amer. Math. Soc., Providence, RI,
  2004.

\bibitem{ReichsteinYoussin2000Essential-dimen}
Zinovy Reichstein and Boris Youssin.
\newblock Essential dimensions of algebraic groups and a resolution theorem for
  {$G$}-varieties.
\newblock {\em Canad. J. Math.}, 52(5):1018--1056, 2000.
\newblock With an appendix by J\'anos Koll\'ar and Endre Szab\'o.

\bibitem{ReichsteinYoussin2002A-birational-in}
Zinovy Reichstein and Boris Youssin.
\newblock A birational invariant for algebraic group actions.
\newblock {\em Pacific J. Math.}, 204(1):223--246, 2002.

\bibitem{ReichsteinYoussin2002Conditions_sati}
Zinovy Reichstein and Boris Youssin.
\newblock Conditions satisfied by characteristic polynomials in fields and
  division algebras.
\newblock {\em J. Pure Appl. Algebra}, 166(1-2):165--189, 2002.

\bibitem{Saltman1982Generic_Galois_}
David~J. Saltman.
\newblock Generic {G}alois extensions and problems in field theory.
\newblock {\em Adv. in Math.}, 43(3):250--283, 1982.

\bibitem{SerganovaSkorobogatov2007Del_Pezzo_surfa}
Vera~V. Serganova and Alexei~N. Skorobogatov.
\newblock Del {P}ezzo surfaces and representation theory.
\newblock {\em Algebra Number Theory}, 1(4):393--419, 2007.

\bibitem{Serre1977Linear-represen}
Jean-Pierre Serre.
\newblock {\em Linear representations of finite groups}.
\newblock Springer-Verlag, New York, 1977.
\newblock Translated from the second French edition by Leonard L. Scott,
  Graduate Texts in Mathematics, Vol. 42.

\bibitem{Serre2008Le-group-de-cre}
Jean-Pierre Serre.
\newblock Le groupe de {C}remona et ses sous-groupes finis.
\newblock {\em S\'eminaire Bourbaki}, 2008.

\bibitem{Serre2009A-Minkowski-sty}
Jean-Pierre Serre.
\newblock A {M}inkowski-style bound for the order of the finite subgroups of
  the {C}remona group of rank 2 over an arbitrary field.
\newblock {\em Moscow Math. J}, 9:193--208, 2009.

\bibitem{Tokunaga2006Two-dimensional}
Hiro-o Tokunaga.
\newblock Two-dimensional versal {$G$}-covers and {C}remona embeddings of
  finite groups.
\newblock {\em Kyushu J. Math.}, 60(2):439--456, 2006.

\bibitem{Wiman1896Zur_Thorie_der_}
Anders Wiman.
\newblock Zur {T}heorie der endlichen {G}ruppen von birationalen
  {T}ransformationen in der {E}bene.
\newblock {\em Math. Ann.}, 48(1-2):195--240, 1896.

\bibitem{Zhang2001Automorphisms-o}
De-Qi Zhang.
\newblock Automorphisms of finite order on rational surfaces.
\newblock {\em J. Algebra}, 238(2):560--589, 2001.
\newblock With an appendix by I. Dolgachev.

\end{thebibliography}

\end{document}